\documentclass[11pt]{article}
\usepackage{amsmath,amsfonts,amssymb,mathrsfs,cases,mathdots,colordvi,url}
\usepackage{enumerate}
\usepackage{amssymb,latexsym}
\usepackage{graphicx}
\usepackage{amstext}
\usepackage{amsmath}
\usepackage{amsopn}
\usepackage{amsbsy}
\usepackage{amscd}
\usepackage{amsxtra}
\usepackage{graphics}
\usepackage{booktabs,multirow}
\usepackage{tabularx}
\usepackage{makecell}
\usepackage{bbding}
\usepackage{pifont}
\usepackage{hyperref}
\usepackage{tikz}
\usepackage{pifont}
\usepackage{amsthm}
\usepackage{graphicx}
\usepackage{subfigure}

\usetikzlibrary{shapes.geometric,arrows}
\usetikzlibrary{positioning, shapes.geometric}
\tikzstyle{block} = [rectangle, draw, fill=white!50,
text width=33em, text centered, rounded corners, minimum height=3em]
\tikzstyle{line} = [draw, -latex']
\tikzstyle{arrow} = [thick,->,>=stealth]

\DeclareMathOperator*{\argmax}{arg\,max}
\numberwithin{equation}{section}

\newtheorem{theorem}{Theorem}[section]
\newtheorem{definition}{Definition}[section]
\newtheorem{proposition}{Proposition}[section]

\newtheorem{lemma}{Lemma}[section]
\newtheorem{corollary}{Corollary}[section]
\newtheorem{remark}{Remark}[section]
\newtheorem{example}{Example}[section]

\begin{document}

\title{Calm local optimality  for nonconvex-nonconcave minimax problems
\thanks{The research of Ye is supported by NSERC. Zhang's work is supported by National Science Foundation of China (No.12222106), Guangdong Basic and Applied Basic Research Foundation (No. 2022B1515020082) and Shenzhen Science and Technology Program (No. RCYX20200714114700072). The alphabetical order of the authors indicates the equal contribution to the paper.}}

\author{Xiaoxiao Ma\thanks{Department of Mathematics and Statistics, University of Victoria, Victoria, B.C., Canada V8W 2Y2. Email: xiaoxiaoma@uvic.ca.}
\and
Wei Yao\thanks{Department of Mathematics and National Center for Applied Mathematics Shenzhen, Southern University of Science and Technology, Shenzhen, Guangdong, China. Email: yaow@sustech.edu.cn.}
\and
Jane J. Ye\thanks{Department of Mathematics and Statistics, University of Victoria, Victoria, B.C., Canada V8W 2Y2. Email: janeye@uvic.ca.}
\and
Jin Zhang\thanks{Corresponding author. Department of Mathematics and SUSTech International Center for Mathematics and National Center for Applied Mathematics Shenzhen, Southern University of Science and Technology, and Peng Cheng Laboratory, Shenzhen, Guangdong, China. Email: zhangj9@sustech.edu.cn.}}
\date{}

\maketitle

\begin{abstract}
\noindent

Nonconvex-nonconcave minimax problems have found numerous applications in various  fields including machine learning. However,  questions remain about what is a good surrogate for local minimax optimum and how to characterize the minimax optimality.  Recently Jin, Netrapalli, and Jordan (ICML 2020) introduced a concept of local minimax point and derived optimality conditions for the smooth and unconstrained case. In this paper, we introduce the concept of calm local minimax point, which is a local minimax point with a calm radius function.  With the extra calmness property we obtain first and second-order sufficient and necessary optimality conditions for a very general class of nonsmooth nonconvex-nonconcave minimax problem. Moreover we show that the calm local minimax optimality and the local minimax optimality coincide under a weak sufficient optimality condition for the maximization problem. This equivalence allows us to derive stronger optimality conditions under weaker assumptions for local minimax optimality.

\noindent {\bf Key words:}\hspace{2mm} minimax problem, local optimality, calmness,  first-order optimality condition, second-order optimality condition \vspace{3mm}

\noindent {\bf AMS Subject Classifications:}\hspace{2mm} 90C26, 90C30, 90C31, 90C33, 90C46, 49J52, 91A65

\end{abstract}


\section{Introduction}
In this paper, we consider the following minimax problem
  \begin{equation}\label{minimax}\tag{Min-Max}
  \min_{x\in X}\max_{y\in  Y} f(x,y),
  \end{equation}
  where the objective function  $f:\mathbb{R}^n\times\mathbb{R}^m\rightarrow\mathbb{R}$ is possibly a nonsmooth function, the nonempty sets $X \subseteq\mathbb{R}^n$, $ Y\subseteq\mathbb{R}^m$ are closed but may be nonconvex.
Throughout the paper, we assume that for each $x\in X$, the inner maximization problem $\max _{y^{\prime} \in  Y} f(x, y^{\prime})$ has an optimal solution.

There are two ways to understand the minimax problem. One way is nonsequential/simultaneous and leads to the concept of a Nash equilibrium. That is,  given $x$, the function $f(x,y')$ is maximized over $y'\in Y$ while given $y$, the function $f(x',y)$ is minimized over $x'\in X$. Another way is sequential
 and leads to the concept of a Stackelberg equilibrium. That is,  for any given $x$, the function $f(x,y')$ is maximized over $y'\in Y$  and the inner maximum value $\max_{y\in Y}f(x',y)$ is minimized over $x'\in X$. In a convex-concave case which means that  $f(x,y)$ is convex in $x$ and concave in $y$ and the sets $X,Y$ are convex, by the celebrated minimax theorem \cite{Sion1958}, the order between the minimization and maximization can be switched and these two concepts are equivalent in the sense that the resulting optimal value of the minimax problems for the two concepts are the same. However for the  nonconvex-nonconcave minimax problem where the objective $f(x,y)$ may be  nonconvex in $x$ and nonconcave in $y$ and $X,Y$ may not be convex, the order between the minimization and maximization is crucial, these two concepts may be different.

Theoretical investigations and numerical algorithms for convex-concave minimax problems have been extensively studied \cite{Demyanov1974,LinJin2020-2,Morgenstern1953,Myerson2013,Nemirovsky1983,Sion1958,ZhangHong2022}.
Nevertheless, recent advances in machine learning, such as generative adversarial networks \cite{Goodfellow2014}, adversarial training \cite{Madry2017}, and reinforcement learning \cite{Omidshafiei2017}, have  introduced new challenges for studying nonconvex-nonconcave minimax optimization. In recent years, there are more and more works for nonconvex-nonconcave minimax problems (e.g., \cite{JNJ20,Li2021,LinJin2020,Luo2022,Thekumparampil2019}).

In most applications of minimax problems, in particular those from the machine learning, the sequential/Stackelberg concept of minimax is required. Denote the value function of the maximization problem as $V(x):=\max_{y \in  Y} f(x,y).$ Then the minimax problem is equivalent to minimizing $V(x)$ over $x\in X$, see e.g. \cite{GuoYe2023} . But the value function is global which means one has to solve a nonconcave maximization problem globally. In practice for solving a nonconvex/nonconcave minimization/maximization problem, many algorithms aim to find stationary points or local optimal solutions as local surrogates for optimal solutions.
What is an appropriate definition for local surrogates of  the sequential/Stackelberg concept of the minimax problem? To answer this question, we should consider some concepts where both $x$ and $y$ are optimized locally. Since we consider problem \eqref{minimax} as a sequential game, i.e., the second player (the one selects $y$) can observe the action taken by the first player (the one selects $x$) and adjust her action accordingly, the radii of the local neighborhoods where maximization or minimization takes place can be different. Based on these considerations, Jin et al. \cite{JNJ20} introduced the concept of local minimax points which is equivalent to saying that a point  $(\bar x,\bar y)$ is a local minimax point if  there exists a radius function $\tau: \mathbb{R}_{+} \rightarrow \mathbb{R}_{+}$ satisfying $\tau(\delta) \rightarrow 0$ as $\delta \downarrow 0$ such that $\bar y$ is a  maximum point of $f\left(\bar x , \cdot\right)$  on $ Y\cap \mathbb{B}_{\delta}(\bar y)$, and $\bar x$ is a  minimum point of the localized optimal value function $V_{\tau(\delta)}(x):=\max_{\{y| \|y-\bar y\|\leq \tau(\delta)\}} f(x,y)$ on $ X\cap \mathbb{B}_{\delta}(\bar x)$, for any small enough $\delta$.
They demonstrated that under mild conditions, all stable limit points of gradient descent ascent (GDA) are exactly local minimax points up to some degenerate points.  Recently, extensions and optimality conditions for local minimax points  for problem \eqref{minimax} has been studied in \cite{DaiZh-2020, JC22, JNJ20,Zhang2022}. Jin et al. \cite{JNJ20} focus on optimality conditions  for the smooth unconstrained case, i.e., $f$ is twice continuously differentiable, $X=\mathbb{R}^n$ and $Y=\mathbb{R}^m$.  Dai and Zhang \cite{DaiZh-2020} study optimality conditions  for the smooth minimax problems with equality and inequality constraints, i.e., $f$ is twice continuously differentiable, $X$ and $Y$ are described by equality and inequality constraints.   Jiang and Chen \cite{JC22} derive optimality conditions for the nonsmooth and constrained case where $X$ and $Y$ are required to be convex. Zhang et al. \cite{Zhang2022} study the smooth constrained case where $X$ and $Y$ are  sometimes assumed to be convex or even the whole space based on the local optimal value function $V_{\epsilon}(x):=\max_{\{y| \|y-\bar y\|\leq \epsilon\}} f(x,y)$.

\subsection{Motivating example}

%
In fair classification \cite{Nouiehed2019}, the objective is to minimize the maximum loss over multiple categories. An example formulation is the problem
$$
\min _{x \in \mathbb{R}^n} \max _{i \in\{1, \ldots, m\}} \ell_i(x),
$$
where $\ell_i(x)$ represents the loss on category $i$ with $x$ denoting neural network parameters. A reformulation \cite{Nouiehed2019} of this problem where $Y$ is a simplex in $\mathbb{R}^m$ is given by the zero-sum game
\begin{equation*}
	\min _{x \in \mathbb{R}^n} \max _{y \in Y} \sum_{i=1}^m y_i \ell_i(x).
\end{equation*}

We observe a subclass of local minimax points in the fair classification problem, whose radius function $\tau$ in Definition \ref{unJNJ} has good properties.

\begin{example}
Consider the fair classification problem
$$
\min _{x \in \mathbb{R}} \max _{y \in [0,1]} \, f(x,y):=y\ell_1(x)+(1-y)\ell_2(x),
$$
where $\ell_1(x)=-x^3+x$ and $\ell_2(x)=-x^3$. Then $(0,0)$ is not a local Nash equilibrium (which means that $\bar x=0$ is not a local minimizer of $f(x, 0)=-x^3$ and $\bar y=0$ is not a local minimizer of $f(0, y)$ over $y\in [0,1]$), since $x=0$ is not a local minimizer of $f(x, 0)=-x^3$. On the other hand, for $0<\delta \leq 1$, we have
$$
\max _{y^{\prime} \in[0, \delta]} f(x, y^{\prime})=\left\{\begin{aligned}
-x^3+\delta x, & \text { if } x \geq 0, \\
-x^3, & \text { if } x<0 .
\end{aligned}\right.
$$
Thus, the point $(0,0)$ is a local minimax point with $\tau(\delta):=\delta$ since for any $0<\delta \leq 1$ and $|x| \leq \delta \leq \sqrt{\delta}$,
$$
f(0, y) \leq f(0,0) \leq \max _{y^{\prime} \in[0, \delta]} f(x, y^{\prime}).
$$
\end{example}

Actually, we find that in many minimax examples presented in the literature, see e.g., \cite[Figure 1]{JNJ20}, \cite[Examples 3.21 and A.1]{JC22}, the radius function $\tau$ is calm at $0$. Moreover, the strict local minimax point defined in \cite[Proposition 20]{JNJ20} and the differential Stackelberg equilibrium defined in \cite[Definition 4]{FCR2020}  also turn out to be local minimax points with a calm radius function. Recall that a function $\phi(x):\mathbb{R}^n\rightarrow \mathbb{R}$ is said to be calm at $\bar x$ if there is $\kappa>0$ and $U$, a neighborhood of $\bar x$,  such that $|\phi(x) -\phi(\bar x)|\leq \kappa \|x-\bar x\|$ for all $x\in U$.

Motivating by the above observations, in this paper, we introduce a new notion of local minimax point, which requires the radius function $\tau$ in the definition of a local minimax point  to be calm at $0$. Note that this is equivalent to  imposing a bound on the ratio of the radii of the local neighborhoods where the  maximization and minimization takes place respectively.   It turns out that this extra calmness property is essential for us to develop optimality conditions for minimax problems.

\subsection{Contributions}
We summarize our main contributions as follows.
\begin{itemize}
\item[(i)] We introduce the concept of the calm local minimax point,  and analyze their relationship with various types of minimax points. In particular we give a general sufficient condition under which the calm local minimax optimality and the local minimax optimality coincide.
The concept of calm local minimax points alleviates the issue of non-existence in local Nash equilibria and properly reflects the order of the minimization and maximization problems.
 Furthermore, since the calm local minimax optimality and the local minimax optimality coincide under the  weak sufficient optimality condition on the maximization problem, our optimality conditions based on the new concept give a more precise and comprehensive characterization of local minimax points than the existing ones.

\item[(ii)] We derive first-order optimality conditions in primal and dual forms for nonsmooth constrained minimax problems, as well as the second-order sufficient and necessary optimality conditions. When reducing the problem to some special cases, e.g., when $f$ is smooth, $X$ and $Y$ are set-constrained systems, we give explicit form for the optimality conditions.  {Our findings effectively capture the nested structure of minimax problems and offer explicit optimality conditions derived for the initial problem data. We have demonstrated that in the case of a smooth problem with some easily satisfied properties of $X$ and $Y$ and the fulfillment of a weak sufficient condition for local optimality in the maximization problem, our necessary conditions can be sharper, and our sufficient condition requires  weaker  assumptions compared to the existing results.} In particular, for the smooth and unconstrained minimax problem, our second-order optimality condition   recovers the one derived by Jin et al. \cite{JNJ20} in the case where $\nabla_{yy}^2 f(\bar x,\bar y)$ is negative definite and is sharper otherwise; for the smooth and constrained case with equality and inequality constraints, our second-order optimality condition is sharper and our assumptions are weaker than the one in Dai and Zhang \cite{DaiZh-2020}.  Unlike those in Zhang et al. \cite{Zhang2022}, all of our optimality conditions are based on the point $(\bar x,\bar y)$ of concern and not requiring any other information such as the local optimal solutions of the  maximization problem on $Y$.
Since we do not require the convexity of the constraint sets $X$ and $Y$, our optimality conditions apply to more general problems than those in Jiang and Chen \cite{JC22}. By using an example, we show that  even in the smooth and unconstrained case, our optimality condition can be used to rule out the possibility of a non-local minimax point while the one in Jiang and Chen \cite{JC22} can not.
\end{itemize}

The remainder of this paper is organized as follows. In Section \ref{prelim}, we introduce the necessary notation and background, including techniques in variational analysis and obtain some preliminary results on optimality conditions for nonsmooth optimization problems. In Section \ref{section-points}, we discuss
different types of minimax points and introduce a new notion for local optimality in the minimax problem -- calm local minimax point. In Section \ref{nonsmooth-section}, we give our new optimality conditions for the general case. In Section \ref{section-compare}, we focus on some special cases and compare our results with existing results in the literature.

\section{Notations and preliminary results}\label{prelim}

In this section, we will first introduce some preliminary materials and results that will be instrumental in deriving optimality conditions for the minimax problem.

{\bf Notations:} We denote by $\mathbb{R}_+^r$ the nonnegative orthant and $\mathbb{\overline R}=[-\infty,\infty]$ the extended real line. For any $z \in \mathbb{R}^r$, $\|z\|$ denotes its Euclidean norm. For $z \in \mathbb{R}^r$ and $\epsilon>0$, we denote by $\mathbb{B}_\epsilon(z):=\{z' \,| \,\|z'-z\|\leq \epsilon\}$ the closed ball centered at $z$ with radius $\epsilon$ and by $\mathbb{B}$ the closed unit ball. For any two vectors $a, b$ in $\mathbb{R}^r$, we denote by $\langle a, b\rangle$ the inner product. For any $z \in \mathbb{R}^r$ and $S \subseteq \mathbb{R}^r$, $\mathrm{dist}(z,S):=\inf_{z'\in S}\|z-z'\|$. For a set $S \subseteq \mathbb{R}^r$, $\delta_S: \mathbb{R}^r \to \mathbb{\overline R}$ denotes the indicator function and $S^{\perp}:=\{\alpha \in \mathbb{R}^r | \langle \alpha, z\rangle=0, \forall z \in S\}$ denotes the orthogonal complement. We denote by $\operatorname{cl} S, \operatorname{co} S$ the closure  and  the convex hull of $S$, and by $\sigma_{S}(\bar z):=\sup_{s\in S} \langle \bar z, s\rangle $ the support function of
$S$ at $\bar z$, respectively. For a set  $S\subseteq \mathbb{R}^r$, a point $\bar z \in \mathbb{R}^r$, and a sequence $z_k$, the notation $z_k \stackrel{S}{\rightarrow} \bar{z}$ means that the sequence $z_k\in S$ goes to $\bar z$. $l(t) = o(t)$ means $l(t)/t \to 0$ as $ t \downarrow 0$. For a set-valued mapping $\Gamma: \mathbb{R}^n \rightrightarrows \mathbb{R}^m$, ${\rm dom\ }\Gamma:=\{x \in \mathbb{R}^n | \Gamma(x) \not= \emptyset\}$ denotes the domain of $\Gamma$, ${\rm gph\ }\Gamma:=\{(x,y) \in \mathbb{R}^n \times \mathbb{R}^m | x \in \mathbb{R}^n, y\in \Gamma(x)\}$ denotes the graph of $\Gamma$.  For a single-valued map $\Phi: \mathbb{R}^r \rightarrow \mathbb{R}$, we denote by $\nabla \Phi(z) \in \mathbb{R}^{r}$ and $\nabla^2 \Phi(z) \in \mathbb{R}^{r \times r}$ the gradient vector of $\Phi$ at $z$ and the Hessian matrix of $\Phi$ at $z$, respectively. For a matrix $A \in \mathbb{R}^{n \times m}, A^T$ is its transpose. For a symmetric matrix $A \in \mathbb{R}^{r \times r},$ $A\prec 0\ (\succeq 0)$ means that the matrix $A$ is a negative definite (positive semidefinite) matrix and $A^{-1}$ is the inverse matrix.

\subsection{Variational analysis}
\begin{definition}[tangent and normal cones \cite{BonSh00,RoWe98}]
Given  $S\subseteq\mathbb{R}^r, \bar z\in S$, the tangent/contingent cone to $S$ at $\bar z$ is defined by \begin{eqnarray*}
 T_S(\bar z)&:=&
  \big\{w \in\mathbb{R}^r \, \big| \, \exists \ t_k\downarrow 0,\; w_k\to w \ \ {\rm with}
 \ \ \bar z+t_k w_k\in S \big\}.
 \end{eqnarray*}
For $w\in T_S(\bar z)$,
the outer second-order tangent set to $S$ at $\bar z$ in direction $w$ is defined by
 \begin{eqnarray*}
 T_S^2(\bar z; w)
 := \left \{ \nu \in \mathbb{R}^r \, \Big| \, \exists \ t_k \downarrow 0, \nu_k \rightarrow \nu \ \ {\rm with}
 \ \  \bar z+t_kw+\frac{1}{2}t^2_k \nu_k \in S
  \right \}.
 \end{eqnarray*}
 The regular/Fr\'echet normal cone, the proximal normal cone, and the limiting/Mordukhovich normal cone to $S$ at $\bar z$ are given, respectively, by
 \begin{eqnarray*}
 \widehat{N}_S(\bar z)&:=&\left\{z^*\in \mathbb{R}^r \, \big| \, \langle z^*, z-\bar z\rangle \le
 o\big(\|z-\bar z\|\big) \ \forall z\in S\right\},\\
{N}^p_S(\bar z)&:=&\left\{z^* \in \mathbb{R}^r \, \big|  \exists \ \gamma>0:   \langle z^*,z-\bar z\rangle\leq \gamma \|z-\bar z\|^2 \ \forall z\in S \right\},\\
 N_S(\bar z)&:=&
 \left \{z^* \in \mathbb{R}^r \, \Big| \, \exists \ z_k \overset{S}{\to} \bar z, z^*_{k}\rightarrow z^* \ {\rm with }\  z^*_{k}\in \widehat{N}_{S}(z_k) \right \}.
\end{eqnarray*}
\end{definition}

The regular normal cone to $S$ at $\bar{z}$ \cite[Proposition 6.5]{RoWe98} can also be characterized by
\begin{equation}\label{normalcone-equ}
\widehat{N}_{S}(\bar{z}):=\left\{z^* \in \mathbb{R}^{r} \mid \langle z^*, w\rangle \leq 0 \quad \forall w\in T_S(\bar z)\right\}.
\end{equation}
For a closed set $S$, one always has $N_{S}^{p}(\bar{z}) \subseteq \widehat{N}_{S}(\bar{z}) \subseteq {N}_{S}(\bar{z})$, where all the cones agree and reduce to the normal cone of convex analysis if $S$ is convex.

\begin{definition}[subderivatives, superderivatives and semiderivatives; Clarke generalized directional derivative and Clarke generalized gradient \text{\cite[Definitions 8.1 and 7.20]{RoWe98}}, \text{\cite[page 10]{Cla-1990}}]
Consider a function $\psi: \mathbb{R}^r \rightarrow \overline{\mathbb{R}}$, a point $\bar z$ with $\psi(\bar z)$ finite, and $w \in  \mathbb{R}^r$.  The subderivative and the superderivative of  $\psi$ at $\bar z$ for  $w$ is defined by
\begin{eqnarray*} {\rm d}\psi(\bar z)(w)&:=& \liminf\limits_{{t\downarrow 0} \atop {w'\to w}} \frac{\psi(\bar{z}+tw')-\psi(\bar{z})}{t},\\
{\rm d}^+\psi (\bar z)(w) &:=& \limsup\limits_{{t\downarrow 0} \atop {w'\to w}} \frac{\psi(\bar z+tw')-\psi(\bar z)}{t},
\end{eqnarray*}
respectively.
When  the limit
$${\rm d}\psi(\bar z)(w)={\rm d}^+\psi (\bar z)(w)=\lim_{{t\downarrow 0} \atop {w'\to w}} \frac{\psi(\bar z+tw')-\psi(\bar z)}{t}$$  exists, we say $\psi$ is semidifferentiable at $\bar z$ for $w$ (or Hadamard differentiable at $\bar z$ in direction $w$). Further if $\psi$ is semidifferentiable at $\bar z$ for every $w$, we say that $\psi$ is semidifferentiable at $\bar z$. It is easy to see that if $\psi$ is semidifferentiable at $\bar z$ for $w$, then
\begin{equation}\label{eqn2.2new}{\rm d}(-\psi)(\bar z)(w)=-{\rm d}\psi(\bar z)(w).\end{equation}
When $\psi$ is Lipschitz continuous, the Clarke generalized directional derivative of $\psi$ at $\bar z$ along the direction $w$ is defined as
$$\psi^\circ(\bar z;w):=\limsup\limits_{{t\downarrow 0} \atop {z'\to \bar z}} \frac{\psi(z'+tw)-\psi(z')}{t}.$$
Moreover, the Clarke generalized gradient is defined as
$$\partial^\circ \psi(\bar z):=\{\xi \in \mathbb{R}^r| \langle \xi, w\rangle \leq \psi^\circ(\bar z;w) \quad \forall w\in \mathbb{R}^r\}.$$
 \end{definition}
  We know from \cite[page 10]{Cla-1990} that
\begin{equation}\label{psicirc}
\psi^\circ(\bar z;w)=\max_{\xi \in \partial^\circ \psi(\bar z)}\langle \xi, w\rangle.
\end{equation}

By \cite[Theorem 7.21]{RoWe98},  if $\psi$  is semidifferentiable at $\bar z$, then ${\rm d} \psi(\bar z)(w)$ is finite for any $w \in \mathbb{R}^r$, $\psi$ is continuous at $\bar z$ and $w\rightarrow {\rm d}\psi(\bar z)(w)$ is homogeneous and continuous. When $\psi$  is semidifferentiable at $\bar z$ for $w$, $\psi$ is also directionally differentiable at $\bar z$ in direction $w$ and the subderivative/superderivative coincides with the classical directional derivative at $\bar z$ in direction $w$. That is,
\begin{equation}{\rm d}\psi(\bar z)(w)={\rm d}^+\psi (\bar z)(w)=\psi'(\bar z; w):=\lim _{t \downarrow 0} \frac{\psi(\bar z+t w)-\psi(\bar z)}{t}.\label{eqn2.2}\end{equation}
 Moreover when $\psi$ is Lipschitz around $\bar z$, the directional differentiability is equivalent to the semidifferentiability.
Hence when $\psi$ is Lipschitz and semidifferentiable around $\bar z$, we have
 $$ \psi^\circ(\bar z;w):=\limsup\limits_{{t\downarrow 0} \atop {z\to \bar z}} \frac{\psi(z+tw)-\psi(z)}{t} \geq \limsup\limits_{{t\downarrow 0} } \frac{\psi(\bar z+tw)-\psi(\bar z)}{t}= {\rm d}\psi(\bar z)(w).$$
 If $\psi$ is continuously differentiable at $\bar z$, it is semidifferentiable at $\bar z$ and for any $w$, one has ${\rm d} \psi(\bar z)(w)=\psi'(\bar z; w)= \nabla \psi (\bar z)^Tw$ \cite[Exercise 8.20]{RoWe98}.

\begin{definition} [second subderivatives, twice semiderivatives, and twice epi-derivatives \text{\cite[Defintions 13.3 and 13.6]{RoWe98}}] Let $\psi: \mathbb{R}^r \rightarrow \overline{\mathbb{R}}$, $\psi(\bar z)$ be finite and $\bar v, w\in  \mathbb{R}^r$. The second subderivative  of $\psi$ at $\bar z$ for $\bar v$  and $w$ is
 \begin{eqnarray*}
 {\rm d}^2\psi(\bar z;\bar v)(w):= \liminf\limits_{{t\downarrow 0} \atop {w'\to w}} \frac{\psi(\bar{z}+tw')-\psi(\bar{z})-t\langle \bar{v}, w'\rangle}{\frac{1}{2}t^2}.
\end{eqnarray*}
On the other hand, the second subderivative  of $\psi$ at $\bar z$ for $w$ (without mention of $\bar v$) is defined by
 \begin{eqnarray}\label{secondsub}
 {\rm d}^2\psi(\bar z)(w):= \liminf\limits_{{t\downarrow 0} \atop {w'\to w}} \frac{\psi(\bar{z}+tw')-\psi(\bar{z})-t {\rm d} \psi(\bar z)(w')}{\frac{1}{2}t^2},
\end{eqnarray}
where the sum of $\infty$ and $-\infty$ is interpreted as $\infty$.
The function $\psi$ is twice semidifferentiable at $\bar z$ if it is semidifferentiable at $\bar z$ and  the `liminf'' in  (\ref{secondsub}) is replaced by the ``lim''
 for any $w \in  \mathbb{R}^r$.
The function $\psi$ is said to be twice epi-differentiable at $\bar z$ for $\bar v$ if for any $w \in \mathbb{R}^r$ and any sequence $t_k\downarrow 0$ there exists a sequence $w_k\rightarrow w$ such that
\begin{equation}\label{Defn2.6}
{\rm d}^2\psi(\bar z;\bar v)(w) = \lim_{k\rightarrow \infty}  \frac{\psi(\bar{z}+t_kw_k)-\psi(\bar{z})-t_k\langle \bar{v}, w_k\rangle}{\frac{1}{2}t_k^2}.
\end{equation}
 \end{definition}
When $\psi$  is twice semidifferentiable at $\bar z$ for $w$, $\psi$ is also twice directionally differentiable at $\bar z$ in direction $w$ and the second subderivative coincides with the classical second directional derivative at $\bar z$ in direction $w$. That is,
 \begin{equation}{\rm d}^2\psi(\bar z)(w)=\psi''(\bar z; w):=\lim _{t \downarrow 0} \frac{\psi(\bar z+t w)-\psi(\bar z) -t \psi'( \bar z;w)}{\frac{1}{2} t^2}. \label{eqn2.5new}\end{equation} It is easy to see that if $\psi$ is twice semidifferentiable at $\bar z$, then
 \begin{equation}\label{twice-epi}
{\rm d}^2 (-\psi)(\bar z)(w)=-{\rm d}^2 \psi(\bar z)(w).
 \end{equation}
  With twice epi-differentiability, the function $\psi$ is properly twice epi-differentiable at $\bar z$ for $\bar v$ if in addition the function ${\rm d}^2\psi(\bar z;\bar v)$ is proper \cite[Definition 13.6]{RoWe98}. By \cite[Exercise 13.7]{RoWe98}, if $\psi$ is twice semidifferentiable at $\bar z$, then ${\rm d}^2 \psi(\bar z)(w)$ is finite for any $w \in \mathbb{R}^r$.
If $\psi$ is twice continuously differentiable at $\bar z$, then it is twice semidifferentiable at $\bar z$ and for $\bar v=\nabla \psi(\bar z)$, and one has ${\rm d}^2\psi(\bar z;\bar v)(w)={\rm d}^2\psi(\bar z)(w)=w^T \nabla^2\psi (\bar z)w$ \cite[Example 13.8]{RoWe98}.

For a function   $\psi(x,y):\mathbb{R}^n \times \mathbb{R}^m \to \overline{\mathbb{R}}$ and $w=(u,h)\in \mathbb{R}^n \times \mathbb{R}^m$, we denote the subderivative of $\psi$ at $(\bar x,\bar y)$  with respect to $x$ for $u$ and $y$ for $h$ by ${\rm d}_x\psi(\bar x,\bar y)(u)$ and  ${\rm d}_y\psi(\bar x,\bar y)(h)$, respectively. Similarly we denote  the second subderivative of $\psi$ at $(\bar x,\bar y)$  with respect to $x$ for $u$ and $y$ for $h$ by ${\rm d}^2_{xx}\psi(\bar x,\bar y)(u)$ and  ${\rm d}^2_{yy}\psi(\bar x,\bar y)(h)$, respectively. When $\psi$ is twice semidifferentiable, it is also semidifferentiable and by (\ref{eqn2.2}) we have
\begin{equation*} {\rm d} \psi(\bar x,\bar y)(0,h) =\lim _{t \downarrow 0} \frac{\psi(\bar x, \bar y+t h)-\psi(\bar x,\bar y)}{t}={\rm d}_y \psi(\bar x,\bar y)(h)\end{equation*} and by (\ref{eqn2.5new})
\begin{equation}\label{2nd-cal}
\begin{aligned}
{\rm d}^2 \psi(\bar x,\bar y)(0,h)
& = \lim\limits_{t\downarrow 0} \frac{\psi(\bar{x},\bar{y}+th)-\psi(\bar{x},\bar y)-t {\rm d} \psi(\bar x,\bar y)(0,h)}{\frac{1}{2}t^2} \\
 & = \lim\limits_{t\downarrow 0} \frac{\psi(\bar{x},\bar{y}+th)-\psi(\bar{x},\bar y)-t {\rm d}_y \psi(\bar x,\bar y)(h)}{\frac{1}{2}t^2}  \\
& = {\rm d}^2_{yy} \psi(\bar x,\bar y)(h).
\end{aligned}
\end{equation}
\begin{definition} Let $\psi(x,y):\mathbb{R}^n \times \mathbb{R}^m \to {\mathbb{R}}$. We say that the separation property holds for the subderivative of $\psi$ at $(\bar x,\bar y)$ if for all $(u,h)\in \mathbb{R}^n \times \mathbb{R}^m$,
\begin{equation}\label{d=dx+dy}
{\rm d} \psi(\bar x,\bar y)(u,h)={\rm d}_x \psi(\bar x,\bar y)(u)+{\rm d}_y \psi(\bar x,\bar y)(h) .
\end{equation}
We say that the separation property holds for the second subderivative of $\psi$ at $(\bar x,\bar y)$ if the separation property holds for the subderivative of $\psi$ at $(\bar x,\bar y)$ and for all $(u,h)\in \mathbb{R}^n \times \mathbb{R}^m$,
\begin{equation}\label{2ndseparation}
{\rm d}^2 \psi(\bar x,\bar y)(u,h)=2{\rm d}_{xy}^2 \psi(\bar x,\bar y)(u,h)+{\rm d}^2_{xx} \psi(\bar x,\bar y)(u)+{\rm d}^2_{yy} \psi(\bar x,\bar y)(h) ,
\end{equation}
where  $${\rm d}_{xy}^2 \psi(\bar x,\bar y)(u,h):= \liminf\limits_{{t\downarrow 0} \atop {u'\to u, h'\to h}} \frac{{\rm d}_y\psi(\bar{x}+tu',\bar{y})(h')-{\rm d}_y\psi(\bar{x},\bar{y})(h')}{t}.$$
\end{definition}

{The separation property will be useful in deriving the optimality conditions. Next, we give a sufficient condition for the separation property.}

\begin{proposition}\label{separation}Suppose that $\psi$ is  semidifferentiable around $(\bar x,\bar y)$. Moreover suppose that either  ${\rm d}_y \psi(x,\bar y)(h)$ is a continuous function of $x$ around $\bar x$ or ${\rm d}_x \psi(\bar x,y)(u)$ is a continuous function of $y$ around $\bar y$. Then the separation property holds for the subderivative of $\psi$ at $(\bar x,\bar y)$.
\end{proposition}
\begin{proof}
Suppose that $\psi$ is  semidifferentiable around $(\bar x,\bar y)$. Moreover suppose that  ${\rm d}_y \psi(x,\bar y)(h)$ is a continuous function of $x$ around $\bar x$. Then we have
\begin{align*}
{\rm d} \psi(\bar x,\bar y)(u,h) &= \lim_{t\downarrow 0} \frac{\psi(\bar{x}+tu,\bar{y}+th)-\psi(\bar{x},\bar y)}{t} \\
& = \lim_{t\downarrow 0} \frac{\psi(\bar{x}+tu,\bar{y}+th)-\psi(\bar{x}+tu,\bar{y})-\psi(\bar{x},\bar{y}+th)+\psi(\bar{x},\bar{y})}{t}\\
&\quad + \lim_{t\downarrow 0} \frac{\psi(\bar{x}+tu,\bar{y})-\psi(\bar{x},\bar{y})}{t}+\lim_{t\downarrow 0} \frac{\psi(\bar{x},\bar{y}+th)-\psi(\bar{x},\bar{y})}{t} \\
&={\rm d}_x \psi(\bar x,\bar y)(u)+{\rm d}_y \psi(\bar x,\bar y)(h).
\end{align*}
The last equality holds by \cite[Theorem 7.21]{RoWe98}, i.e.,
\begin{eqnarray*}
\psi(\bar{x}+tu,\bar{y}+th)-\psi(\bar{x}+tu,\bar{y}) &=& {{\rm d}_y \psi(\bar x+tu,\bar y)(th)
 }+o(t)=t{\rm d}_y \psi(\bar x+tu,\bar y)(h)+
 o(t),\\
 \psi(\bar{x},\bar{y}+th)-\psi(\bar{x},\bar{y}) &=&  {{\rm d}_y \psi(\bar x,\bar y)(th)}+
 o(t)=t{\rm d}_y \psi(\bar x,\bar y)(h)+
 o(t).\end{eqnarray*}
 Similarly, if $\psi$ is semidifferentiable around $(\bar x,\bar y)$ and   ${\rm d}_x \psi(\bar x,y)(u)$ is a continuous function of $y$, then
 $${\rm d} \psi(\bar x,\bar y)(u,h)={\rm d}_x \psi(\bar x,\bar y)(u)+{\rm d}_y \psi(\bar x,\bar y)(h).$$
 \end{proof}

 \begin{proposition}
 Suppose that $\psi$ is twice semidifferentiable around $(\bar x,\bar y)$, the separation property holds for the subderivative of $\psi$ at $(\bar x,\bar y)$, and ${\rm d}^2_{xy} \psi(x,\bar y)(u,h)={\rm d}^2_{yx} \psi(x,\bar y)(u,h)$. Suppose further that either  ${\rm d}^2_{yy} \psi(x,\bar y)(h)$ is a continuous function of $x$ around $\bar x$ or ${\rm d}^2_{xx} \psi(\bar x,y)(u)$ is a continuous function of $y$ around $\bar y$. Then the separation property holds for the second subderivative of $\psi$ at $(\bar x,\bar y)$.
\end{proposition}
\begin{proof}
Suppose that $\psi$ is twice semidifferentiable around $(\bar x,\bar y)$, the separation property holds for the subderivative of $\psi$ at $(\bar x,\bar y)$, and ${\rm d}^2_{xy} \psi(x,\bar y)(u,h)={\rm d}^2_{yx} \psi(x,\bar y)(u,h)$. Moreover suppose that ${\rm d}^2_{yy} \psi(x,\bar y)(h)$ is a continuous function of $x$ around $\bar x$. Then
  \begin{align*}
&\quad\ {\rm d}^2 \psi(\bar x,\bar y)(u,h) \\ &= \lim_{t\downarrow 0} \frac{\psi(\bar{x}+tu,\bar{y}+th)-\psi(\bar{x},\bar y)-t{\rm d}\psi(\bar x,\bar y)(u,h)}{\frac{1}{2} t^2} \\
& = \lim_{t\downarrow 0} \frac{\psi(\bar{x}+tu,\bar{y}+th)-\psi(\bar{x}+tu,\bar{y})-\psi(\bar{x},\bar{y}+th)+\psi(\bar{x},\bar{y})}{\frac{1}{2}t}\\
&\quad + \lim_{t\downarrow 0} \frac{\psi(\bar{x}+tu,\bar{y})-\psi(\bar{x},\bar{y})-t{\rm d}_x\psi(\bar x,\bar y)(u)}{\frac{1}{2}t^2}+\lim_{t\downarrow 0} \frac{\psi(\bar{x},\bar{y}+th)-\psi(\bar{x},\bar{y})-td_y\psi(\bar x,\bar y)(h)}{\frac{1}{2}t^2} \\
&= \lim_{t\downarrow 0} \frac{{\rm d}_y\psi(\bar{x}+tu,\bar{y})(th)+\frac{1}{2}t^2 {\rm d}_{yy}^2 \psi(\bar x+tu, \bar y)(h)+o(t^2)-{\rm d}_y\psi(\bar{x},\bar{y})(th)-\frac{1}{2} t^2{\rm d}_{yy}^2 \psi(\bar x, \bar y)(h)-o(t^2)}{\frac{1}{2} t^2}\\
&\quad +{\rm d}^2_{xx} \psi(\bar x,\bar y)(u)+{\rm d}^2_{yy} \psi(\bar x,\bar y)(h)\\
&=2{\rm d}_{xy}^2 \psi(\bar x,\bar y)(u,h)+{\rm d}^2_{xx} \psi(\bar x,\bar y)(u)+{\rm d}^2_{yy} \psi(\bar x,\bar y)(h).
\end{align*}
The third equality holds by \cite[Exercise 13.7 and Proposition 13.5]{RoWe98}, i.e.,
\begin{eqnarray*}
\psi(\bar{x}+tu,\bar{y}+th)-\psi(\bar{x}+tu,\bar{y}) &=& {\rm d}_y\psi(\bar{x}+tu,\bar{y})(th)+\frac{1}{2} t^2{\rm d}_{yy}^2 \psi(\bar x+tu, \bar y)(h)+o(t^2),\\
 \psi(\bar{x},\bar{y}+th)-\psi(\bar{x},\bar{y}) &=&  {\rm d}_y\psi(\bar{x},\bar{y})(th)+\frac{1}{2} t^2 {\rm d}_{yy}^2 \psi(\bar x, \bar y)(h)+o(t^2).\end{eqnarray*}
 Similar discussions can be given under the assumption that  ${\rm d}^2_{xx} \psi(\bar x,y)(u)$ is a continuous function of $y$ around $\bar y$.
\end{proof}

The computation of subderivatives is crucial for applying the optimality conditions. In cases where $\psi$ is a composition of certain special functions, we provide the following chain rule for computing subderivatives which will be used in Example \ref{example-non} . In fact one can easily extend the result to the  more general case where $\psi(x_1, \dots, x_n):=\varphi(g(x_1),\dots,g(x_n))$, where $\varphi$ is a $C^2$ function.

\begin{proposition}\label{chainrule} Let $\psi(x,y):=\varphi(g(x),g(y))$ where $\varphi(\alpha,\beta):\mathbb{R}\times \mathbb{R} \rightarrow \mathbb{R}$ is $C^2$ and $g:\mathbb{R}^n \rightarrow \mathbb{R}$ is Lipschitz continuous and directionally differentiable around $\bar x$ and $\bar y$. Then $\psi$ is Lipschitz continuous, twice semidifferentiable at $(\bar x,\bar y)$ in any direction $(u,h)$. Moreover, the separation property holds for
the subderivative of $\psi$ at $(\bar x,\bar y)$, and
\begin{eqnarray*}
{\rm d}\psi(\bar x,\bar y)(u,h)&=&\nabla \varphi(g(\bar x),g(\bar y))^T (g'(\bar x;u),g'(\bar y;h)),\\
{\rm d}_x\psi(\bar x, \bar y)(u)&=&\nabla_\alpha \varphi (g(\bar x),g( \bar y))g'(\bar x;u),\\
{\rm d}_y\psi(\bar x, \bar y)(h)&=&\nabla_\beta \varphi (g(\bar x),g(\bar y))g'(\bar y;h),\\
{\rm d}^2\psi(\bar x,\bar y)(u,h)&=&(g'(\bar x;u),g'(\bar y;h))^T \nabla^2 \varphi(g(\bar x),g(\bar y))(g'(\bar x;u),g'(\bar y;h)),\\
{\rm d}_{yy}^2\psi(\bar x,\bar y)(h)&=&{g'(\bar y;h)^T \nabla_{yy}^2 \varphi(g(\bar x),g(\bar y))g'(\bar y;h).}
\end{eqnarray*}
\end{proposition}
\begin{proof}
Since
 \begin{eqnarray*}
&& \frac{\psi(\bar x+tu',\bar y +th')-\psi(\bar x,\bar y)}{t}\\
&=& \frac{\varphi(g(\bar x+tu'),g(\bar y +th'))-\varphi(g(\bar x),g(\bar y))}{t}\\
&=& \frac{\nabla \varphi(g(\bar x),g(\bar y))^T (g(\bar x+tu')-g(\bar x),g(\bar y +th'))-g(\bar y))}{t}+\frac{o(t)}{t},
 \end{eqnarray*}
 the limit exists when $t\downarrow 0, u'\rightarrow u, h'\rightarrow h$,
 $\psi$ is also directionally differentiable at $(\bar x,\bar y)$ and the formula for ${\rm d}\psi(\bar x,\bar y)(u,h)$ is obtained.
 Similarly for each $x,y$ around $(\bar x,\bar y)$,
\begin{eqnarray*}
{\rm d}_x\psi(\bar x, y)(u)&=&\nabla_\alpha \varphi (g(\bar x),g( y))g'(\bar x;u),\\
{\rm d}_y\psi(x, \bar y)(h)&=&\nabla_\beta \varphi (g(x),g(\bar y))g'(\bar y;h).
\end{eqnarray*}
Hence ${\rm d}_x\psi(\bar x, y)(u)$ and ${\rm d}_y\psi(x,\bar y)(h)$ are continuous in $y, x$ respectively. Therefore by Proposition \ref{separation}, the separation property holds at $(\bar x,\bar y)$. Since
  \begin{eqnarray*}
&& \frac{\psi(\bar x+tu',\bar y +th')-{\psi}(\bar x,\bar y)-t{\rm d}{\psi}(\bar x,\bar y)(u,h)}{\frac{1}{2}t^2}\\
&=& \frac{\varphi(g(\bar x+tu'),g(\bar y +th'))-\varphi(g(\bar x),g(\bar y))-t\nabla \varphi(g(\bar x),g(\bar y))^T (g'(\bar x;u),g'(\bar y;h))}{\frac{1}{2}t^2}\\
&=& \frac{\nabla \varphi(g(\bar x),g(\bar y))^T (g(\bar x+tu')-g(\bar x),g(\bar y +th')-g(\bar y))- t(g'(\bar x;u),g'(\bar y;h))}{\frac{1}{2}t^2}+\frac{o(t^2)}{\frac{1}{2}t^2}\\
&& +\frac{1}{2}\frac{(g(\bar x+tu')-g(\bar x),g(\bar y +th')-g(\bar y))^T \nabla^2 \varphi(g(\bar x),g(\bar y)) (g(\bar x+tu')-g(\bar x),g(\bar y +th')-g(\bar y))}{\frac{1}{2}t^2},
 \end{eqnarray*}
 and the limit exists when $t\downarrow 0, u'\rightarrow u, h'\rightarrow h$,
 $\psi$ is twice directionally differentiable at $(\bar x,\bar y)$ and the formula for
 ${\rm d}^2\psi(\bar x,\bar y)(u,h)$
 is obtained.
 {Similarly,
 \begin{eqnarray*}
 {\rm d}_{yy}^2\psi(\bar x,\bar y)(h)&=&g'(\bar y;h)^T \nabla_{yy}^2 \varphi(g(\bar x),g(\bar y))g'(\bar y;h).
\end{eqnarray*}}
 \end{proof}

 By Definition \ref{Defn2.6}, for a set $S \subseteq \mathbb{R}^r, \bar z\in S,$ and $\bar v,w \in \mathbb{R}^r$, the second subderivative of the indicator function $\delta_S$ at $\bar z$ for
 $\bar v$ and $w$ is
  \begin{equation}\label{indicator}
\begin{aligned}
 {\rm d}^2\delta_S(\bar z;\bar v)(w):= \liminf\limits_{{t\downarrow 0} \atop {w'\to w}} \frac{\delta_S(\bar{z}+tw')-t\langle \bar{v}, w'\rangle}{\frac{1}{2}t^2}
= \liminf\limits_{{t\downarrow 0,w'\to w} \atop {\bar{z}+tw' \in S}} \frac{-2\langle \bar{v}, w'\rangle}{t}.
\end{aligned}
 \end{equation}

 \begin{proposition}[\text{\cite{BenkoGfrerer2022}}] Consider a closed set $S \subseteq \mathbb{R}^r, \bar z\in S,$ and $\bar v,w \in \mathbb{R}^r$. The following statement hold:
\begin{itemize}
\item[(i)] If $w\not \in T_{S}(\bar z)$ or $\langle \bar v,w\rangle <0, $ then ${\rm d}^2\delta_S(\bar z;\bar v)(w)=\infty$.
\item[(ii)] We have
$${\rm d}^2\delta_S (\bar z;\bar v)(w)\leq -\sigma_{T^2_S(\bar z; w)}(\bar  v)$$
if and only if $w\in T_S(\bar z)$ and $\langle \bar v, w\rangle \geq 0$ or $T_S^2(\bar z; w)=\emptyset,$ where $\sigma_{T_{S}^{2}(\bar{z}; w)}(\bar v)$ denotes the support function of the second-order tangent set $T_{S}^{2}(\bar{z}; w)$.
\end{itemize}
\end{proposition}

Consider a function $\psi: \mathbb{R}^r \rightarrow \mathbb{R}$. If  $\psi$ is differentiable, then by taking $\bar v$ as $\nabla \psi(\bar z)$ in (\ref{indicator}), we have
$${\rm d}^2\delta_S(\bar z; \nabla \psi(\bar z))(w)= \liminf\limits_{{t\downarrow 0} \atop {w'\to w}} \frac{\delta_S(\bar{z}+tw')-t\langle \nabla \psi(\bar z), w'\rangle}{\frac{1}{2}t^2}=\liminf\limits_{{t\downarrow 0,w'\to w} \atop {\bar{z}+tw' \in S}} \frac{-2\langle \nabla \psi(\bar z), w'\rangle}{t}
.$$
In order to deal with a nonsmooth function $\psi$,  we propose the following definition.

\begin{definition}\label{Defn2.7}
Let $S \subseteq \mathbb{R}^r$, $\psi: \mathbb{R}^r \rightarrow \mathbb{R}$ be semidifferentiable at $\bar z \in S$.
 The second subderivative  of $\delta_{S}$ at $\bar z$ for $ {\rm d}\psi(\bar z)$   and $w$ is
$${\rm d}^2 \delta_S(\bar z; {\rm d}\psi(\bar z))(w):= \liminf\limits_{t \downarrow 0, w' \to w} \frac{\delta_{S}(\bar z+tw')-t{\rm d}\psi(\bar z)(w')}{\frac{1}{2}t^2}= \liminf\limits_{{t\downarrow 0,w'\to w} \atop {\bar{z}+tw' \in S}} \frac{-2{\rm d}\psi(\bar z)(w')}{t},$$
where the sum of $\infty$ and $-\infty$ is interpreted as $\infty$.
We say that $\delta_{S}$ is twice epi-differentiable at $\bar z \in S$ for ${\rm d}\psi(\bar z)$ if for any $w \in \mathbb{R}^r$ and any sequence $t_k\downarrow 0$ there exists a sequence $w_k\rightarrow w$ such that
\begin{equation*}\label{}
	{\rm d}^2 \delta_S(\bar z;{\rm d}\psi(\bar z))(w) = \lim_{k\rightarrow \infty}  \frac{\delta_{S}(\bar z+t_kw_k)- t_k{\rm d}\psi(\bar z)(w_k)}{\frac{1}{2}t_k^2}.
\end{equation*}
\end{definition}Suppose that $\psi$ is semidifferentiable at $\bar z$ for $w$, then we have
 $\lim_{w'\rightarrow w} {\rm d}\psi(\bar z)(w')={\rm d}\psi (\bar z)(w).$
It follows  that \begin{equation}\label{infinity}
 {\rm d}^2 \delta_S(\bar z; {\rm d}\psi(\bar z))(w)=\infty \qquad \mbox{ if }  w\not \in T_{S}(\bar z) \mbox{ or } {\rm d}\psi(\bar z)(w)<0.
 \end{equation}

We now introduce  parabolic properties of a set.
\begin{definition}[parabolic properties of sets \cite{ABM21}]
Let  $S \subseteq \mathbb{R}^{r}$ be  nonempty. $S $ is said to be  parabolically derivable at $\bar{z} \in \mathbb{R}^{r}$ for $w \in \mathbb{R}^{r}$ if $T_{S}^{2}(\bar{z}; w) \neq \emptyset$ and for each $\nu \in T_{S}^{2}(\bar{z}; w)$ there exist a number $\varepsilon>0$ and an arc $\xi:[0, \varepsilon] \rightarrow S$ such that $\xi(0)=\bar{z}, \xi_{++}^{\prime}(0)=w$, and $\xi_{++}^{\prime \prime}(0)=\nu$ with
$$
\xi_{++}^{\prime}(0):=\lim _{t \downarrow 0} \frac{\xi(t)-\xi(0)}{t},\quad \xi_{++}^{\prime \prime}(0):=\lim _{t \downarrow 0} \frac{\xi(t)-\xi(0)-t \xi_{++}^{\prime}(0)}{\frac{1}{2} t^{2}} .
$$
$S $ is said to parabolically regular at $\bar{z} \in S$ for $\bar{v} \in \mathbb{R}^{r}$ if for any $w \in \mathbb{R}^{r}$ with $\mathrm{d}^{2} \delta_{S}(\bar{z}; \bar{v})(w)<\infty$ there exist, among all the sequences $t_{k} \downarrow 0$ and $w_{k} \rightarrow w$ satisfying the condition
$$
\frac{\delta_{S}(\bar z +t_kw_k)-\delta_{S}(\bar z)-t_k \langle \bar v,w_k\rangle}{\frac{1}{2} t^{2}}   \rightarrow \mathrm{d}^{2} \delta_{S}(\bar{z}; \bar{v})(w) \text { as } k \rightarrow \infty,
$$
those with the additional property that
$
\limsup _{k \rightarrow \infty} \frac{\left\|w_{k}-w\right\|}{t_{k}}<\infty .
$
\end{definition}
The following results show that a set that has parabolic properties has useful properties.
\begin{proposition}[\text{\cite[Theorems 3.3 and 3.6]{ABM21}}]\label{Tm3.6ABM}
Let $S$ be a closed subset of $\mathbb{R}^{r}$ with $\bar{z} \in S$, and let $\bar v \in N_{S}^{p}(\bar{z})$. Assume further that $S$ is parabolically derivable at $\bar{z}$ for every vector $w \in T_S({\bar z}) \cap \{\bar v\}^{\perp}$. If $S$ is parabolically regular at $\bar{z}$ for $\bar v$, then the indicator function $\delta_{S}$ is properly twice epi-differentiable at $\bar z$ for $\bar v$, its second subderivative is finite for all $w \in T_S({\bar z}) \cap \{\bar v\}^{\perp}$ and is calculated by $\mathrm{d}^{2} \delta_{S}(\bar{z}; \bar v)(w)=-\sigma_{T_{S}^{2}(\bar{z}; w)}(\bar v)$.
\end{proposition}
The following sets are parabolically derivable at $\bar{z}$ for every vector $w \in T_S({\bar z}) \cap \{\bar v\}^{\perp}$ and are parabolically regular at $\bar{z}$ for any $\bar v \in N_{S}^{p}(\bar{z})$: the convex polyhedral set \cite[Example 3.4]{ABM21}, the disjunctive set \cite[Remark 3.5]{ABM21}, the second-order cone \cite[Example 5.8]{ABM21}, and the cone of positive semidefinite symmetric matrices \cite[Theorem 6.2]{ABM21}.

Now, we consider the constraint system
\begin{equation}\label{constraints} S=\{z \in \mathbb{R}^r|g(z)\in \Sigma\},
\end{equation}
where $g:\mathbb{R}^r \rightarrow \mathbb{R}^{q}$ and  $\Sigma \subseteq \mathbb{R}^q$  is closed.

\begin{definition}[metric subregularity constraint qualification]
Let $\bar z\in S$ where $S$ is the constraint system defined by (\ref{constraints}). We say that the metric subregularity constraint qualification (MSCQ)  for  $S$ holds at  $\bar z$ if there exist a neighborhood $U$ of $\bar z$ and a constant $\rho>0$ such that
$${\rm dist}(z,S) \leq \rho\ {\rm dist}(g(z),\Sigma) \qquad \forall z\in U.$$
\end{definition}
Sufficient conditions for MSCQ of the equality and inequality system can be found in \cite[Theorem 7.4]{YeZhou18}, e.g., the first-order sufficient condition for metric subregularity (FOSCMS), the second-order sufficient condition for metric subregularity (SOSCMS), the Mangasarian-Fromovitz constraint qualification (MFCQ), and the linear constraint qualification, i.e., $g$ is affine and $\Sigma$ is the union of finitely many polyhedral convex sets.

Under the MSCQ, the parabolic properties of the set $S$ are guaranteed under the parabolic properties of the set $\Sigma$ provided that $\Sigma$ is convex.  Note that a convex set satisfying the following conditions (i)-(iii) with $g(z)=z$ includes a polyhedral convex set, the second-order cone and a  semidefinite matrix cone.

\begin{proposition}[\text{\cite[Theorems 4.5 and 5.6, Propositions 4.2 and 5.2]{ABM21}}]\label{system}
Let $\bar z\in S$ where $S$ is the constraint system defined by (\ref{constraints}) and $\bar v \in N_S(\bar z)$. Suppose that
$g$ is twice continuously differentiable, the MSCQ holds for $S$ at $\bar z$, and
\begin{itemize}
\item[(i)] $\Sigma$ is convex,
\item[(ii)] $\Sigma$ is parabolically derivable at $g(\bar z)$ for all vectors $\nabla g(\bar z)\eta $ satisfying $\nabla g(\bar z) \eta \in T_\Sigma(g(\bar z)) $ and $\eta\in \{\bar v\}^{\perp}$,
\item[(iii)] $\Sigma$ is parabolically regular at $g(\bar z)$ for every $\lambda \in \Lambda(\bar z,\bar v):=\{\lambda \in N_{\Sigma}(g(\bar z))|\nabla g(\bar z)\lambda =\bar v\}$.
\end{itemize}
Then, the set $S$ is parabolically regular at $\bar z$ for $\bar v$ and is parabolically derivable at $\bar z$ for all vectors  $w\in T_S(\bar z) \cap \{\bar v\}^\perp$. Moreover, $N_{S}^{p}(\bar z)=\widehat{N}_{S}(\bar z)=N_{S}(\bar z)$,
\begin{equation}\label{linearizedcone1}
T_S(\bar z)=\{w \in \mathbb{R}^r|\nabla g(\bar z)w \in T_{\Sigma}(g(\bar z))\},
\end{equation}
and for any $w\in T_S(\bar z) \cap \{\bar v\}^\perp$, the second subderivative is finite and
\begin{equation}\label{2ndsub-chain}
{\rm d}^2 \delta_{S}(\bar z;\bar v)(w) = \max_{\lambda \in \Lambda(\bar z,\bar v)} \{ \langle \lambda, \nabla^2 g(\bar z)(w,w) \rangle + {\rm d}^2 \delta_{\Sigma}(g(\bar z),\lambda)(\nabla g(\bar z)w)\}.
\end{equation}
\end{proposition}
Let $\bar z \in S$, $\Sigma:=\mathbb{R}^p_-\times \{0\}^q$, and suppose that the MSCQ holds for $S$ at $\bar z$. We have then that
\begin{equation}\label{linearizedcone}
T_S(\bar z)=\left\{w \in \mathbb{R}^r: \begin{array}{l}
\nabla g_i(\bar z)^T w=0, i=1, \ldots, q, \\
\nabla g_i(\bar z)^T w \leq 0, i \in I(\bar z)
\end{array}\right\},
\end{equation}
where
$$
I(\bar z):=\left\{i: g_i(\bar z)=0, i=1, \ldots, p\right\}
$$
denotes the set of inequality constraints active at $\bar z$.
For $w \in T_S(\bar z)$,
\begin{equation}\label{2ndlinearizedcone}
T_S^2(\bar z,w)=\left\{\nu \in \mathbb{R}^r: \begin{array}{l}
\nabla g_i(\bar z)^T \nu+w^T\nabla^2 g_i(\bar z)w=0, i=1, \ldots, q, \\
\nabla g_i(\bar z)^T \nu+w^T\nabla^2 g_i(\bar z)w \leq 0, i \in I_1(\bar z,w)
\end{array}\right\},
\end{equation}
where
$$
I_1(\bar z,w):=\left\{i\in I(\bar z): \nabla g_i(\bar z)^Tw=0\right\}.
$$
Let $S$ be convex polyhedral. Consider $\bar z \in S$, $\bar v \in N_{S}(\bar{z})$, and $w \in T_S(\bar z)$. By \cite[Exercise 13.17]{RoWe98}, we have ${\rm d}^2 \delta_{S}(\bar z;\bar v)(w)=0$ for any $w \in T_S({\bar z}) \cap \{\bar v\}^{\perp}$.
By \cite[Proposition 13.12]{RoWe98},
\begin{equation}\label{poly-2ndtangent}
T_S^2(\bar z;w)=T_{T_S(\bar z)}(w)=T_S(\bar z)+\mathbb{R}w.
\end{equation}

When the set $S$ has the parabolic properties, the second subderivative of the indicator function $\delta_S$ can be given.

\begin{proposition}[calculation of the second subderivative of the indicator function]\label{Calculation}
Given $S \subseteq \mathbb{R}^r$, $\bar z \in S$, and $\bar v \in N_{S}(\bar{z})$.
\begin{itemize}
\item[(i)] Suppose that $S=\mathbb{R}^r$. Then,
\begin{equation*}\label{Calfor2nd-uncons}
{\rm d}^2 \delta_{S}(\bar z;\bar v)(w)=\begin{cases}0 & \text { if } w \in \{\bar v\}^{\perp}, \\ \infty & \text { otherwise. }\end{cases}
\end{equation*}
\item[(ii)] Suppose that $S$ is convex polyheral. Then,
\begin{equation*}\label{Calfor2nd}
{\rm d}^2 \delta_{S}(\bar z;\bar v)(w)=\begin{cases}0 & \text { if } w \in T_S({\bar z}) \cap \{\bar v\}^{\perp}, \\ \infty & \text { otherwise. }\end{cases}
\end{equation*}
\item[(iii)] Suppose that $S:=\{z\in \mathbb{R}^r| g(z) \in \Sigma\}$ where $S$ and $\Sigma$ satisfy the conditions in Proposition \ref{system}. Then, for any $w \in T_S({\bar z}) \cap \{\bar v\}^{\perp}$, $\mathrm{d}^{2} \delta_{S}(\bar{z}; \bar v)(w)$ is finite and
\begin{eqnarray*}
\lefteqn{
\mathrm{d}^{2} \delta_{S}(\bar{z}; \bar v)(w)}\\
&& = \begin{cases}\max_{\lambda \in \Lambda(\bar z,\bar v)} \left \{ \langle \lambda, \nabla^2 g(\bar z)(w,w) \rangle -\sigma_{T_{\Sigma}^{2}(g(\bar z),\nabla g(\bar z)w)}(\lambda) \right \} & \text { if } w \in T_S({\bar z}) \cap \{\bar v\}^{\perp}, \\ \infty & \text { otherwise. }\end{cases}
\end{eqnarray*}
Here, the tangent cone $T_S({\bar z})$ can be characterized by \eqref{linearizedcone1}.
\end{itemize}
\end{proposition}
\begin{proof}
(i) and (ii) can be derived by \cite[Exercise 13.17]{RoWe98}. (iii) follows from Proposition \ref{Tm3.6ABM} and \eqref{2ndsub-chain}.
\end{proof}

\subsection{Optimality conditions for nonsmooth optimization problems}
The following first-order conditions for optimality for unconstrained problems can be found in \cite[Proposition 3.99]{BonSh00}.
 \begin{proposition}[first-order optimality conditions  for unconstrained problems]\label{thm2} Let  $\psi: \mathbb{R}^r \rightarrow \overline{\mathbb{R}}$ and $\bar z$ be a point where $\psi(\bar z )$ is finite. Then:
 \begin{itemize}
 \item[(a)] If $\bar z$ is a local minimum point of the function $\psi$, then
\begin{equation*}
{\rm d} \psi(\bar z)(w) \geq 0\quad \forall w\in \mathbb{R}^r.
\end{equation*}
      \item[(b)] {If ${\rm d} \psi(\bar z)(w) > 0$ for all $w\in \mathbb{R}^r\setminus \{0\}$, then $\bar z$ is a local minimum point of the function $\psi$.}
 \end{itemize}
  \end{proposition}

The following result was given in \cite[Proposition 3.100]{BonSh00}; see also  \cite[Theorem 13.24]{RoWe98}.

 \begin{proposition}[second-order optimality conditions for unconstrained problems]\label{thm1} Let  $\psi: \mathbb{R}^r \rightarrow \overline{\mathbb{R}}$ and $\bar z$ be a point where $\psi(\bar z )$ is finite. Then:
 \begin{itemize}
 \item[(a)] If $\bar z$ is a local minimum point of the function $\psi$, then
\begin{equation*}
{\rm d}^2 \psi(\bar z;0)(w)\geq 0\ \ \forall w\in \mathbb{R}^r.
\end{equation*}
      \item[(b)] {The condition ${\rm d}^2 \psi(\bar z;0)(w) > 0$ for all $w\in \mathbb{R}^r\setminus \{0\}$ holds if and only if the second-order growth condition holds at $\bar z$, i.e.,  there exist
     $  \varepsilon>0$ and $\eta>0$ such that
\begin{equation*}
\psi(z)\geq \psi(\bar z)+\varepsilon\|z-\bar z\|^2 \ \ \mbox{ when } \|z-\bar z\|\leq \eta,
\end{equation*}
which implies that $\bar z$ is a local minimum point of the function $\psi$.} \end{itemize}
  \end{proposition}

{Based on the aforementioned results, we can obtain optimality conditions for the constrained problems, which are essential for analyzing the optimality conditions of the minimax problem.}

\begin{proposition}[first-order optimality conditions for constrained problems]\label{non-optimality}
Let $\psi: \mathbb{R}^r \to \mathbb{R}$, $S \subseteq \mathbb{R}^r$ be nonempty and closed and $\bar z \in S$.
\begin{itemize}
\item[(i)]Suppose that $\psi$ is either Lipschitz continuous around $\bar z$  or  semidifferentiable at $\bar z$. If $\bar z$ is a local minimizer of $\psi$ on $S$, then ${\rm d} \psi (\bar z)(w) \geq 0$ for any $w \in T_{S}(\bar z)$.
\item[(ii)]{If ${\rm d}\psi (\bar z)(w) > 0$ for any $w \in T_{S}(\bar z) \setminus \{0\}$, then $\bar z$ is a local minimizer of $\psi$ on $S$.}
\end{itemize}
\end{proposition}
\begin{proof}
(i) When $\psi$ is Lipschitz continuous around $\bar z$, we have $\partial^{\infty} \psi (\bar z) =\{0\}$ \cite[Theorem 1.22]{M18}. The desired result follows from \cite[Theorem 8.15]{RoWe98}.
Now suppose that $f$ is semidifferentiable at $\bar z \in S$. Then for any $w \in T_S(\bar z)$,
\begin{equation*}
	\begin{aligned}
		{\rm d}\psi(\bar z)(w) &  = \lim\limits_{t \downarrow 0, w' \to w} \frac{\psi(\bar z+tw')-\psi(\bar z)}{t} \\
		& = \lim\limits_{{t \downarrow 0, w' \to w} \atop {\bar z+tw' \in S} } \frac{\psi(\bar z+tw')-\psi(\bar z)}{t} \geq 0.
	\end{aligned}
\end{equation*}

(ii) Given $\psi : \mathbb{R}^r \to \mathbb{R}$ and $\bar z \in S \subseteq \mathbb{R}^r$. For any $w \in T_S(\bar z)$ and $v \in \mathbb{R}^r$,
\begin{eqnarray}\label{dpsi+delta-1}
	\begin{aligned}
		{\rm d}(\psi+\delta_S)(\bar z)(w)  & =  \liminf\limits_{t \downarrow 0, w'\to w} \frac{\psi(\bar z+tw')-\psi(\bar z)+\delta_{S}(\bar{z}+tw^{\prime})}{t}\\
		& =  \liminf\limits_{{t \downarrow 0, w'\to w} \atop {\bar{z}+tw^{\prime} \in S}} \frac{\psi(\bar z+tw')-\psi(\bar z)}{t} \geq {\rm d}\psi (\bar z)(w).
	\end{aligned}
\end{eqnarray}

For any $w \not\in T_S(\bar z)$, \begin{eqnarray}\label{dpsi+delta-2}
	{\rm d}(\psi+\delta_S)(\bar z)(w) = \infty.
\end{eqnarray}

The desired result follows from Theorem \ref{thm2} (b), \eqref{dpsi+delta-1} and \eqref{dpsi+delta-2}.
\end{proof}

We now derive second-order optimality conditions for constrained minimization problems. These kinds of results can be traced back to Penot \cite[Theorems 1.2 and 1.7]{Penot1994}. For the smooth problems, under more conditions on set $S$, second-order optimality conditions with ${\rm d}^2 \delta_S(\bar z;-\nabla\psi(\bar z))(w)$ replaced by the support function of the second-order tangent set $\sigma_{T_S^2(\bar z;w )}(\nabla\psi(\bar z))$ have been established in  Bonnans and Shapiro \cite[Sections 3.2.2 and 3.3.3]{BonSh00}. The term $\sigma_{T_S^2(\bar z;w )}(\nabla\psi(\bar z))$ is referred to a ``sigma'' term and it can not be dismissed in general if the set $S$ has some curvature.
\begin{proposition}[second-order  optimality conditions for constrained problems]\label{non-optimality2}
Let $\psi: \mathbb{R}^r \to \mathbb{R}$, $S \subseteq \mathbb{R}^r$ is nonempty and closed and $\bar z \in \mathbb{R}^r$. Suppose that $\psi$ is twice semidifferentiable at $\bar z$.
\begin{itemize}
\item[(i)] If $\bar z$ is a local minimizer of $\psi$ on $S$, then {${\rm d} \psi(\bar z)(w) \geq 0$ for any $w \in T_S(\bar z)$,} and for any { $w \in T_S(\bar z) \cap \{w'|{\rm d} \psi(\bar z)(w') = 0\} $},
$${\rm d}^2 \psi(\bar z)(w)+{\rm d}^2 \delta_S(\bar z;-{\rm d}\psi(\bar z))(w) \geq 0.$$
\item[(ii)] Suppose that {${\rm d} \psi(\bar z)(w) \geq 0$ for any $w \in T_S(\bar z)$,} and that for any $w \in T_S(\bar z) \cap \{w'|{\rm d} \psi(\bar z)(w') = 0\} \setminus \{0\}$,
\begin{equation}\label{y_localmax-non}
{\rm d}^2 \psi(\bar z)(w) +{\rm d}^2 \delta_S(\bar z;-{\rm d}\psi(\bar z))(w)>0.
\end{equation}
Then, $\bar z$ is a local minimizer of $\psi$ on $S$ {with the second-order growth condition, i.e.,  there exist
     $  \varepsilon>0$ and $\eta>0$ such that
\begin{equation*}
\psi(z)\geq \psi(\bar z)+\varepsilon\|z-\bar z\|^2 \ \ \mbox{ when } z \in S\cap \mathbb{B}_\eta(\bar z).
\end{equation*}}
\end{itemize}
\end{proposition}
\begin{proof}
(i) The first-order necessary optimality condition follows from Proposition \ref{non-optimality}.
Let $w \in T_S(\bar z) \cap \{w'|{\rm d} \psi(\bar z)(w') = 0\} $. Since $\psi$ is twice semidifferentiable at $\bar z \in S$, we have
\begin{equation}\label{d2psi+delta}
	\begin{aligned}
		{\rm d}^2(\psi+\delta_{S})(\bar z;0)(w) &  = \liminf\limits_{t \downarrow 0, w' \to w} \frac{\psi(\bar z+tw')-\psi(\bar z)+\delta_{S}(\bar z+tw')}{\frac{1}{2}t^2} \\
		& = \liminf\limits_{t \downarrow 0, w' \to w} \frac{\psi(\bar z+tw')-\psi(\bar z)-t {\rm d}\psi(\bar z)(w')+\delta_{S}(\bar z+tw')+t{\rm d}\psi(\bar z)(w')}{\frac{1}{2}t^2} \\
		& ={\rm d}^2 \psi(\bar z)(w) +{\rm d}^2 \delta_S(\bar z;-{\rm d}\psi(\bar z))(w).
	\end{aligned}
\end{equation}

The desired result follows from \eqref{d2psi+delta} and Theorem \ref{thm1} (a).

(ii) By assumption, ${\rm d} \psi(\bar z)(w) \geq 0$ for any $w \in T_S(\bar z)$.
Let $w\in\mathbb{R}^r$.   {If $w\not\in T_{S}(\bar z)$ or $w \in T_S(\bar z)$ but $ {\rm d} \psi(\bar z)(w) > 0$, then we have ${\rm d}^2 \delta_S(\bar z;-{\rm d}\psi(\bar z))(w) = \infty$ by \eqref{infinity}. By \eqref{d2psi+delta}, it follows that
$${\rm d}^2(\psi+\delta_{S})(\bar z;0)(w) =\infty > 0 \quad \mbox{ either } w\not\in T_{S}(\bar z) \mbox { or } {\rm d} \psi(\bar z)(w) > 0.$$}
If $w\in T_{S}(\bar z) \cap \{w'|{\rm d} \psi(\bar z)(w') = 0\} \setminus \{0\}$, \eqref{d2psi+delta} and \eqref{y_localmax-non} imply ${\rm d}^2(\psi+\delta_{S})(\bar z;0)(w) > 0$.
Finally we obtain
$${\rm d}^2(\psi+\delta_{S})(\bar z;0)(w) > 0 \quad \forall w\in\mathbb{R}^r \setminus \{0\}.$$
Applying Theorem \ref{thm1} (b), we obtain that $\bar z$ is a local minimizer of $\psi$ on $S$.
\end{proof}

\section{Concepts of optimality for the minimax problem}\label{section-points}

In this section, we  introduce a new notion for local optimality of the minimax problem \eqref{minimax}. Before that, we first review several existing optimality concepts for the minimax problem.

\begin{definition}[global minimax point/Stackelberg equilibrium]\label{globaldef}
A point $\left(\bar x, \bar y\right) \in X \times Y$ is a  global minimax point or a Stackelberg equilibrium of problem \eqref{minimax}, if for any $(x, y)$ in $X \times  Y$,
$$
f(\bar x, y) \leq f(\bar x, \bar y) \leq \max _{y^{\prime} \in  Y} f(x, y^{\prime}).
$$
\end{definition}
Define the (global) value function as
$ V(x):=\max_{y \in  Y} f(x,y)$.
Then
 $\left(\bar x, \bar y\right)$ is a global minimax point if and only if $\bar y$ is a global maximum point of $f\left(\bar x , \cdot\right)$  on $ Y$, and $\bar x$ is a global minimum point of $V(x)$ on $X$.

\begin{definition}[saddle point/Nash equilibrium]\label{nash}
A point $\left(\bar x, \bar y\right) \in X \times Y$ is a saddle point or a Nash equilibrium of problem \eqref{minimax}, if for any $(x, y)$ in $X \times  Y$,
$$
f(\bar x, y) \leq f(\bar x, \bar y) \leq f(x, \bar y).
$$
\end{definition}

\begin{definition}[local saddle point/local Nash equilibrium]\label{localnash}
A point $\left(\bar x, \bar y\right) \in X \times Y$ is a local saddle point or a local Nash equilibrium of problem \eqref{minimax} if there exists $\delta >0$ such that for any $x\in X\cap \mathbb{B}_\delta(\bar x), y\in Y\cap \mathbb{B}_\delta(\bar y)$,
$$
f (\bar x, y ) \leq f (\bar x, \bar y) \leq f(x, \bar y).
$$
\end{definition}

\begin{definition}[local minimax point \text{\cite[Definition 14]{JNJ20}}]\label{unJNJ}
A point $(\bar{x}, \bar{y}) \in X \times  Y$ is a local minimax point of problem \eqref{minimax}, if there exist a $\delta_0>0$ and a radius function $\tau: \mathbb{R}_{+} \rightarrow \mathbb{R}_{+}$ satisfying $\tau(\delta) \rightarrow 0$ as $\delta \downarrow 0$, such that for any $\delta \in\left(0, \delta_0\right]$ and any $x\in X\cap \mathbb{B}_\delta(\bar x), y\in Y\cap \mathbb{B}_\delta(\bar y)$, we have
\begin{equation}\label{localmm}
f(\bar{x}, y) \leq f(\bar{x}, \bar{y}) \leq \max_{y' \in  Y\cap\mathbb{B}_{\tau( \delta)}(\bar y)} f (x, y' ).
\end{equation}
\end{definition}
Given $\epsilon \geq 0$ and $\bar y\in Y$, define the local optimal value function at $\bar y$ by
\begin{equation} \label{Vepsilon}
V_{\epsilon}(x):=V_{\epsilon}(x;\bar y):=\max_{y  \in  Y\cap \mathbb{B}_{\epsilon}(\bar y)}f(x,y).
\end{equation}
Note that
$ V_{\epsilon}(\bar x)=V_{\epsilon}(\bar x;\bar y)=f(\bar x,\bar y)$ when $\bar y$ is a maximum point of $f\left(\bar x , \cdot\right)$ in $Y\cap \mathbb{B}_{\epsilon}(\bar y)$.
So $\left(\bar x, \bar y\right)$ is a local minimax point if and only if there exists a function $\tau: \mathbb{R}_{+} \rightarrow \mathbb{R}_{+}$ satisfying $\tau(\delta) \rightarrow 0$ as $\delta \downarrow 0$ such that $\bar y$ is a  maximum point of $f\left(\bar x , \cdot\right)$  on $ Y\cap \mathbb{B}_{\delta}(\bar y)$, and $\bar x$ is a  minimum point of the local optimal value function $V_{\tau(\delta)}(x)$ on $ X\cap \mathbb{B}_{\delta}(\bar x)$, for any small enough $\delta$.

Under the continuity assumption of $f$, Jin et al. \cite[Lemma 16]{JNJ20} showed that $(\bar{x}, \bar{y})\in X\times Y$ is a  local minimax point of problem \eqref{minimax} if and only if $\bar y$ is a local maximum point of the function $f(\bar x, \cdot)$ on $Y$, and
$\bar x$ is a local minimum point of the function $V_{ \epsilon}(x)$ on $X$ for all $\epsilon \downarrow 0$.  Later on, the continuity assumption on $f$ is removed by Zhang et. al. \cite[Proposition 3.9]{Zhang2022}.

It has been shown in Jin et. al. \cite[Remark 15]{JNJ20} that Definition \ref{unJNJ} remains equivalent even if we further restrict function $\tau$ in Definition \ref{unJNJ} to be either monotonic or continuous. The concept of local minimax points is a celebrated idea for characterizing the local optimality for the minimax problems. However, when we tried to characterize  the local optimality, we found  it  important to consider the calmness property of the function $\tau$. Thus, we introduce the following new concept for the local optimality which is  the local minimax point with the function $\tau$ being calm at $0$.

\begin{definition}[calm local minimax point]\label{calmlocal}
A point $(\bar{x}, \bar{y}) \in X \times  Y$ is a calm local minimax point of problem \eqref{minimax} if it is a local minimax point of problem \eqref{minimax} with the  radius function $\tau$ being calm at $0$.
\end{definition}

In Definition \ref{unJNJ} of local minimax points, the radius $R_y$ of the neighborhood of $\bar y$ for maximizing $f(x, y)$ in $y$ is $\tau(\delta)$, i.e., the radius $R_y$ in the local optimal value function \eqref{Vepsilon} is $\tau(\delta)$.
Concurrently the radius $R_x$ of the neighborhood of $\bar x$ for minimizing $V_{\tau(\delta)}(x)$ is $\delta$. Thus the ratio $R_y/R_x$ of these two radii is $\tau(\delta)/\delta$.
In Definition \ref{calmlocal}, the calmness of $\tau$ at $0$ implies that for any sufficiently small $\delta$, we have $\tau(\delta) \leq \kappa \delta$ for some $\kappa > 0$. This is equivalent to imposing a bound on the ratio $R_y/R_x$ of the radii of the local neighborhoods where the maximization and minimization takes place respectively.
Next we can give an equivalent definition of calm local minimax point, without introducing a radius function $\tau$.

\begin{proposition}[an equivalent definition of calm local minimax point]\label{equiv-calm-local}
 The point $(\bar{x},\bar{y})\in X\times Y$ is a calm local minimax point, if and only if there exist a $\delta_0>0$ and a $\kappa>0$, such that for any $\delta\in(0, \delta_0]$ and any $x\in X\cap \mathbb{B}_{\delta}(\bar{x})$, $y\in Y\cap \mathbb{B}_{\delta}(\bar{y})$, we have
			\begin{equation}\label{equiv}
				f(\bar{x}, y) \leq f(\bar{x},\bar{y})
				\leq \max_{y'\in Y\cap \mathbb{B}_{\kappa \delta}(\bar{y})} f(x, y').
			\end{equation}
\end{proposition}
\begin{proof}
``$\Rightarrow$" Let $(\bar{x},\bar{y})\in X\times Y$ be a calm local minimax point. Then, the relation in \eqref{localmm} holds with $\tau$ being calm at $0$, i.e., $\tau(\delta) \leq \kappa\delta$ for some $\kappa >0$. Thus, $\max_{y'\in Y\cap \mathbb{B}_{\tau(\delta)}(\bar{y})} f(x, y') \leq \max_{y'\in Y\cap \mathbb{B}_{\kappa \delta}(\bar{y})} f(x, y')$ and \eqref{equiv} holds.

``$\Leftarrow$" This is obvious since we can let $\tau(\delta)=\kappa \delta$.
\end{proof}

\begin{remark}		 The concept of local minimax optimality was extended to the case where the constraint set  $Y$ depends on variable $x$ in \cite[Definition 1.1]{DaiZh-2020}. Similarly we can also define the calm local minimax optimality for the minimax problem $\min_{x \in X}\max_{y\in Y(x)} f(x,y)$ by requiring  the radius function  defined in \cite[Definition 1.1]{DaiZh-2020} to be calm.
\end{remark}

The following example demonstrates that the concept of calm local minimax optimality is not equivalent to the concept of local minimax optimality.
\begin{example}\label{ex3.1}
Consider
$$
\min _{x \in \mathbb{R}} \max _{y \in \mathbb{R}} f(x, y):=-x^2+2xy^3-y^6 .
$$
We show that $(0,0)$ is a local (and also global) minimax point but not a calm local minimax point.
\begin{itemize}
\item{} Take $\tau(\delta)=\delta^{\frac{1}{3}}$ and $\delta_0=1$. Then for any $|x| \leq \delta$ and $|y| \leq \delta$ with $\delta \in(0,\delta_0]$ we have
$$
-y^6=f(0, y) \leq f(0,0) \leq \max _{y \in \mathbb{B}_{\tau( \delta)}(0)}-x^2+2xy^3-y^6=0.
$$
 Hence $(0,0)$ is a local minimax point.
\item{} For each $x$, we have
$$
-y^6=f(0, y) \leq f(0,0) \leq \max _{y}-x^2+2xy^3-y^6=0,
$$
where $x^{\frac{1}{3}}$ is the maximizer. Hence $(0,0)$ is a global minimax point.
{
\item{}
By Proposition \ref{equiv-calm-local}, one can check that $(0,0)$ is not a calm local minimax point.
Indeed, it is easy to see that $f(0,0)=0$, and for all positive $\delta$, $\kappa$,
$$
\max _{y\in \mathbb{B}_{\kappa\delta}(0)} f(x, y)=\left\{\begin{aligned}
	0, & \text { if } |x| \leq \kappa^3\delta^3, \\
	-x^2+2|x|\kappa^3\delta^3-\kappa^6\delta^6<0, & \text { if } |x| >\kappa^3\delta^3.
\end{aligned}\right.
$$
Since $\delta^3=o(\delta)$, for any given $\kappa$, there always exists sufficiently small $\delta$ and some $x$ such that $\kappa^3\delta^3 <x \leq \delta$. Thus, there exists no $\kappa$ in the equivalent definition in Proposition \ref{equiv-calm-local}.
}
\end{itemize}
\end{example}

Next, we study the relation between the global minimax point and the calm local minimax point. Generally, a global minimax point may not be a local minimax point {(and thus not a calm local minimax point)}; cf. \cite[Proposition 21]{JNJ20} for an explicit example, where the global minimax point can be neither local minimax nor a stationary point {(i.e., $\nabla f(\bar x,\bar y)=0$ in the unconstrained smooth case)}.
Nevertheless, global minimax points can be guaranteed to be calm local minimax if the optimal solution set of the inner maximization problem satisfies the following inner calmness condition.

\begin{definition}[inner calmness \text{\cite[Definition 2.2]{BGO19}}]\label{innercalm}
Consider  a set-valued map $\Gamma: \mathbb{R}^n \rightrightarrows \mathbb{R}^m$.
 Given $\bar x\in X$ and $\bar y \in  \Gamma(\bar x)$, we say that the set-valued map $\Gamma$ is inner calm at $(\bar{x}, \bar{y})$ w.r.t. $X$ if there exist $\kappa >0$ and  $\delta_0>0$  such that
$$\bar y \in \Gamma(x)+\kappa \|x-\bar x\| \mathbb{B} \quad \forall x\in \mathbb{B}_{\delta_0}(\bar x) \cap X,$$
or equivalently (\cite[Lemma 2.2]{BGO19} or \cite[Definition 2.2]{Benko-2021}), if there exists $\kappa >0$ such that for any $x_k \rightarrow \bar x$ with $x_k \in X$ there exists a sequence $y_k$ satisfying $y_k \in \Gamma (x_k )$ and for sufficiently large $k$,
 	$\|y_k-\bar y\| \leq \kappa \|x_k-\bar x\|.$
\end{definition}

Using the inner calmness, the following result gives a condition under which a global minimax point is always a calm local minimax point.

\begin{proposition}\label{JNJ-ST}
Let $(\bar x, \bar y)$ be a global minimax point of problem \eqref{minimax}. If the solution mapping $S(x):=\argmax\limits_{y\in  Y} f( x,y)$ is inner calm at $(\bar x, \bar y)$ w.r.t. $X$, then $(\bar x, \bar y)$ is a calm local minimax point.
\end{proposition}
\begin{proof}
	By the inner calmness of $S$ at $(\bar x, \bar y)$, there exist $\kappa>0$ and $\delta_0>0$ such that for all $x$ in $\mathbb{B}_{\delta_0}(\bar x)\cap X$ there exists $\bar y(x) \in S(x)$ satisfying
	$
	\|\bar{y}(x)-\bar y\| \leq \kappa \|x-\bar x\|
	$.
This implies the existence of $\bar y(x) \in S(x)$ in $\mathbb{B}_{\kappa \delta}(\bar y)\cap Y$ for any $\delta \in\left(0, \delta_0\right]$ and any $x$ in $\mathbb{B}_{\delta}(\bar x)\cap X$. It follows that
	$$\max_{y \in  \mathbb{B}_{\kappa \delta}(\bar y)\cap Y} f (x, y ) = \max _{y \in  Y} f (x , y ) .$$
Thus,  a global minimax point $(\bar x, \bar y)$ is a calm local minimax point with $\tau(\delta)=\kappa \delta$.
\end{proof}

%
%
In \cite[Proposition 6.1]{Chen2023}, Chen et al. gave conditions under which the solution mapping is Lipschitz continuous, and thus is inner calm. We use their result to give the following corollary.

\begin{corollary}\label{corollary-strong-concave}
Suppose that $Y$ is convex and $f$ satisfies the following properties:
\begin{itemize}
\item[(i)] the function $f$ is $L_f$-smooth w.r.t. $x$, that is,
$$
\left\|\nabla_y f(x, y)-\nabla_y f\left(x^{\prime}, y\right)\right\| \leq L_f \left\|x-x^{\prime}\right\| \quad \forall x, x^{\prime} \in   X, y\in Y,
$$
\item[(ii)] {and} it is gradient dominant in $y$ for some $\alpha>0$, i.e.,
$$
L_f
\left\|
y-P_{Y}\left(y+(1 / L_f) \nabla_y f(x, y)\right)
\right\|
\geq \alpha\left\|y-y_p(x)\right\|
 \quad \forall x \in X, y \in Y,
$$
where $P_{Y}(\cdot)$ is the projection operator and $y_p(x)$ is the projection of $y$ onto the solution set $S(x):=\argmax _{y \in Y} f(x, y)$.
\end{itemize}
Then, the solution mapping $S$ is inner calm w.r.t. $X$ at any point $(\bar x,\bar y) \in \operatorname{gph} S$, and thus a global minimax point is a calm local minimax point. Particularly, by \cite[Example 6.1]{Chen2023},
if $Y$ is convex, $f$ is $L_f$-smooth, and $f(x, \cdot)$ is strongly concave for any $x\in X$, then a global minimax point is a calm local minimax point.
\end{corollary}

Note that the condition (ii) in Corollary \ref{corollary-strong-concave} for the unconstrained case (i.e., $Y = \mathbb{R}^m$) is the error bounds of Luo and Tseng \cite{LT1993}, and is equivalent to the well-known Polyak-Łojasiewicz condition, which is weaker than the strong concavity by \cite[Theorem 2]{KNS16}.

{From the definitions, it is clear that a calm local minimax point must be a local minimax point. Next, we give conditions under which they are equivalent. To this end, we need some equivalent descriptions for the local minimax point and the calm local minimax point.

\begin{lemma}\label{locamm-equi}
The point $(\bar{x}, \bar{y}) \in X \times Y$ is a local minimax point, if and only if, there exist a $\delta_0>0$ and a function $\tau: \mathbb{R}_{+} \rightarrow \mathbb{R}_{+}$satisfying $\tau(\delta) \rightarrow 0$ as $\delta \downarrow 0$, such that for any $\delta \in\left(0, \delta_0\right]$ and any $x \in X \cap \mathbb{B}_\delta(\bar{x}), y \in Y \cap \mathbb{B}_\delta(\bar{y})$, we have
$$
f(\bar{x}, y) \leq f(\bar{x}, \bar{y}) \leq \max _{y^{\prime} \in Y \cap \mathbb{B}_{\tau(\|x-\bar{x}\|)}(\bar{y})} f\left(x, y^{\prime}\right).
$$
\end{lemma}
\begin{proof}
``$\Leftarrow$" Similar to the proof for \cite[Remark 15]{JNJ20}, we have the facts that $\tau (\delta)\leq \widetilde\tau(\delta) :=\sup_{\delta' \in (0,\delta]}\tau(\delta')$ for any $\delta \in (0,\delta_0]$, $\widetilde\tau(\delta) \rightarrow 0$ as $\delta \downarrow 0$, and $\widetilde\tau$ is monotonic. Thus, without loss of generality, we assume that $\tau$ is monotonic. Since $\|x-\bar{x}\| \leq \delta$ for $x \in X \cap \mathbb{B}_\delta(\bar{x})$ and $\tau$ is monotonic, we get $\tau(\|x-\bar{x}\|) \leq \tau(\delta)$. Thus
$$
\max _{y^{\prime} \in Y \cap \mathbb{B}_{\tau(\|x-\bar{x}\|)}(\bar{y})} f\left(x, y^{\prime}\right) \leq \max _{y^{\prime} \in Y \cap \mathbb{B}_{\tau(\delta)}(\bar{y})} f\left(x, y^{\prime}\right).
$$
``$\Rightarrow$" For any given $x \in X \cap \mathbb{B}_\delta(\bar{x})$, denote $\delta^{\prime}=\|x-\bar{x}\|$, then $\delta^{\prime} \leq \delta \leq \delta_0$. Thus, by \eqref{localmm} in which $\delta$ be $\delta^{\prime}$, we have
$$
f(\bar{x}, \bar{y}) \leq \max _{y^{\prime} \in Y \cap \mathbb{B}_{\tau\left(\delta^{\prime}\right)}(\bar{y})} f\left(x, y^{\prime}\right)=\max _{y^{\prime} \in Y \cap \mathbb{B}_{\tau(\|x-\bar{x}\|)}(\bar{y})} f\left(x, y^{\prime}\right)
$$
which finishes the proof.
\end{proof}

\begin{lemma}\label{calmlocamm-equi}
The point $(\bar x, \bar y)$ is a calm local minimax point of problem \eqref{minimax}, if and only if, it is a local minimax point and the optimal solution mapping
\begin{equation*}
S_{\tau(\cdot)} (x):=\argmax _{y \in Y \cap \mathbb{B}_{\tau(\|x-\bar{x}\|)}(\bar{y})} f(x, y)
\end{equation*}
is inner calm at $(\bar{x}, \bar{y})$.
\end{lemma}
\begin{proof}
``$\Rightarrow$" Since  $\tau$ is calm at $0$, there exists $\kappa >0$ such that for any $x_k \rightarrow \bar x$, $y_k \in S_{\tau(\cdot)} (x_k )$, and sufficiently large $k$,
 	$$\|y_k-\bar y\| \leq \tau(\|x_k-\bar x\|) \leq \kappa \|x_k-\bar x\|.$$
 By Definition \ref{innercalm}, $S_{\tau(\cdot)}(x)$ is inner calm at $(\bar x,\bar y)$.

``$\Leftarrow$" By the inner calmness of $S_{\tau(\cdot)}(x)$ at $(\bar x, \bar y)$, there exist $\kappa >0$ and  $\delta_0>0$  such that
$$\bar y \in S_{\tau(\cdot)}(x)+\kappa \|x-\bar x\| \mathbb{B} \quad \forall x\in \mathbb{B}_{\delta_0}(\bar x) \cap X.$$

This means that for any $0<\delta \leq \delta_0$ and any $x \in \mathbb{B}_{\delta}(\bar x)\cap X$, there exists $\bar y(x) \in S_{\tau(\cdot)}(x)$ satisfying $\|\bar{y}(x)-\bar y\| \leq \kappa \|x-\bar x\| \leq \kappa \delta$. Hence, with Lemma \ref{locamm-equi}, we have
	$$f(\bar x,y) \leq f(\bar x,\bar y) \leq \max _{y \in Y \cap \mathbb{B}_{\tau(\|x-\bar{x}\|)}(\bar{y})} f (x , y )=f(x,\bar y(x)) \leq \max_{y \in  Y \cap \mathbb{B}_{\kappa \delta}(\bar y)} f (x, y ).$$

Thus, $(\bar x, \bar y)$ is a calm local minimax point with $\tau(\delta)=\kappa \delta$.
\end{proof}

With the above equivalent description for the calm local minimax point, we can now give conditions under which the calm local minimax point is the same as the local minimax point.
}

\begin{proposition}\label{local-to-lip}
Let $(\bar x,\bar y)$ be a local minimax point of problem \eqref{minimax}. Suppose that $f$ is twice semidifferentiable at $(\bar x,\bar y)$. If for any $h\in T_{ Y}(\bar y)
 \setminus \{0\}$,
\begin{equation}\label{2nddescent}
{\rm d}^2_{yy} f(\bar x,\bar y)(h) -{\rm d}^2 \delta_{X \times Y}((\bar x,\bar y);{\rm d}f(\bar x,\bar y))(0,h)<0,
\end{equation}
 then $(\bar x,\bar y)$ is a calm local minimax point.
\end{proposition}
\begin{proof}
Suppose that $(\bar x,\bar y)$ is a local minimax point. Then  there exist a $\delta_0>0$ and a function $\tau: \mathbb{R}_{+} \rightarrow \mathbb{R}_{+}$ satisfying $\tau(\delta) \rightarrow 0$ as $\delta \rightarrow 0$, such that for any $\delta \in\left(0, \delta_0\right]$ and any $x \in X$ satisfying $\|x-\bar{x}\| \leq \delta$, we have
\begin{equation}\label{contra}
f(\bar{x}, \bar{y}) \leq \max _{y' \in  Y\cap \mathbb{B}_{\tau( \delta)}(\bar y)} f (x, y^{\prime} ).
\end{equation}
By Lemma \ref{calmlocamm-equi}, to prove that $(\bar x,\bar y)$ is a calm local minimax point, we only need to show that there exists $\kappa>0$ such that for any $x_k \to \bar x$ with $x_k \in X$ and any $y_k \in S_{\tau(\cdot)}(x_k),$
we have $\|y_k-\bar y\| \leq \kappa \|x_k-\bar x\|$.
To the contrary, suppose that for some $x_k \to \bar x$ with $x_k \in X$, there exist $y_k \in S_{\tau(\cdot)}(x_k)$ such that $\|y_k-\bar y\|/\|x_k-\bar x\|$ diverges.

First, since $\|y_k - \bar y\|\leq \tau(\|x_k-\bar x\|)$ for any $k$
, we have $y_k \to \bar y$. Let $h_k:=(y_k-\bar y)/\|y_k-\bar y\|$ for all $k \in \mathbb{N}$. Then $\|h_k\|=1$, and thus without loss of generality, we can assume that $h_k \rightarrow h$. Set $t_k: =\|y_k-\bar y\|$ for all $k \in \mathbb{N}$, we have $t_k \downarrow 0$ and $y_k = \bar y + t_k h_k$, which implies $h \in T_{Y}(\bar y) \setminus \{0\}$.
Denote $u_k :=(x_k - \bar x)/\|y_k-\bar y\|$, then $x_k=\bar x +t_k u_k$. Since $\|y_k-\bar y\|/\|x_k-\bar x\| \rightarrow \infty$, we have $\|u_k\|=\|x_k-\bar x\|/\|y_k-\bar y\| \to 0$ and thus $u_k \to 0=u$.

By Proposition \ref{non-optimality} (i), since $\bar y$ is a local maximizer of $f(\bar x,y)$ on $Y$, ${\rm d}_y f (\bar x, \bar y )(h') \leq 0$ for any $h'\in T_Y({\bar y})$.
Since $(\bar x, \bar y)$ is a local minimax point,
\begin{eqnarray*}\label{}
\begin{aligned}
0 & \leq \limsup\limits_{k \to \infty} \frac{f (\bar{x}+t_k u_k, \bar{y}+t_kh_k )-f(\bar{x}, \bar{y})}{\frac{1}{2}t_k^2}  \qquad \mbox{ by (\ref{contra})}, \\
& =\limsup\limits_{k \to \infty} \frac{f (\bar{x}+t_k u_k, \bar{y}+t_kh_k )-f(\bar{x}, \bar{y})-t_k{\rm d}f(\bar{x}, \bar{y})(u_k,h_k)+t_k{\rm d}f(\bar{x}, \bar{y})(u_k,h_k)}{\frac{1}{2}t_k^2} \\
& = {\rm d}^2 f(\bar x,\bar y)(0,h)
+ \limsup\limits_{k \to \infty} \frac{t_k{\rm d}f(\bar{x}, \bar{y})(u_k,h_k)}{\frac{1}{2}t_k^2}  \qquad \mbox{since $f$ is twice semidifferentiable }\\
& \leq {\rm d}^2 f(\bar x,\bar y)(0,h) + \limsup\limits_{{t \downarrow 0,(u',h') \to (0,h)} \atop {\bar x+tu'\in X, \bar y+th'\in Y}} \frac{t{\rm d} f(\bar{x}, \bar{y})(u',h')}{\frac{1}{2}t^2} \\
& = {\rm d}^2 f(\bar x,\bar y)(0,h) - \liminf\limits_{{t \downarrow 0,(u',h') \to (0,h)} \atop {\bar x+tu'\in X, \bar y+th'\in Y}} \frac{-t{\rm d} f(\bar{x}, \bar{y})(u',h')}{\frac{1}{2}t^2} \\
&= {\rm d}^2_{yy} f(\bar x,\bar y)(h) -{\rm d}^2 \delta_{X \times Y}((\bar x,\bar y);{\rm d}f(\bar x,\bar y))(0,h),
\end{aligned}
\end{eqnarray*}
where the last equality follows from \eqref{2nd-cal}, a contradiction to \eqref{2nddescent}.
\end{proof}

{When $f$ is differentiable and $X$ is convex polyhedral, the condition \eqref{2nddescent} will reduce to the weak sufficient condition for the local optimality. }

\begin{corollary}\label{local-to-lip2}
Let $(\bar x,\bar y)$ be a local minimax point of problem \eqref{minimax}. Suppose that $f$ is twice semidifferentiable at $(\bar x,\bar y)$, {the separation property holds for the subderivative of $f$ at $(\bar x,\bar y)$,} and for any $h\in T_{ Y}(\bar y) \cap \{h'|{\rm d}_y f(\bar x,\bar y)(h')=0\}\setminus \{0\}$,
\begin{equation}\label{2nddescentnew}
{\rm d}^2_{yy} f(\bar x,\bar y)(h) -{\rm d}^2 \delta_{X }(\bar x;{\rm d}_xf(\bar x,\bar y))(0) -{\rm d}^2 \delta_{Y }(\bar y;{\rm d}_yf(\bar x,\bar y))(h)<0,
\end{equation}
and either ${\rm d}^2 \delta_{X }(\bar x;{\rm d}_xf(\bar x,\bar y))(0)$ or ${\rm d}^2 \delta_{Y }(\bar y;{\rm d}_yf(\bar x,\bar y))(h)$ is finite, then $(\bar x, \bar y)$ is a calm local minimax point. Particularly, if $f$ is differentiable and $X$ is convex polyhedral, then ${\rm d}^2 \delta_{X }(\bar x;{\rm d}_xf(\bar x,\bar y))(0)=0$ and the condition (\ref{2nddescentnew}) is the weak sufficient condition for local optimality of the problem $\max_{y\in Y}f(\bar x,y)$ at $\bar y$.
\end{corollary}
\begin{proof}
Since the separation property holds for the subderivative of $f$ at $(\bar x,\bar y)$, by (\ref{d=dx+dy})
we have
\begin{eqnarray*}
{\rm d}^2 \delta_{X\times Y}((\bar x,\bar y);{\rm d}f(\bar x,\bar y))(u,h)&=& \liminf\limits_{{t\downarrow 0,(u',h')\to (u,h)} \atop {\bar{x}+tu' \in X, \bar{y}+th' \in Y}} \frac{-2{\rm d}_xf(\bar x,\bar y)(u')-2{\rm d}_yf(\bar x,\bar y)(h')}{t}\\
&\geq & \liminf\limits_{{t\downarrow 0,u'\rightarrow u} \atop {\bar{x}+tu' \in X}} \frac{-2{\rm d}_x f(\bar x,\bar y)(u')}{t} +\liminf\limits_{{t\downarrow 0,h'\rightarrow h} \atop {\bar{y}+th' \in Y}} \frac{-2{\rm d}_y f(\bar x,\bar y)(h')}{t}\\
&=& {\rm d}^2 \delta_{X}(\bar x;{\rm d}_x f(\bar x,\bar y))(u)+{\rm d}^2 \delta_{Y}(\bar y;{\rm d}_y f(\bar x,\bar y))(h),
\end{eqnarray*}
where the inequality holds since the assumption that either ${\rm d}^2 \delta_{X }(\bar x;{\rm d}_xf(\bar x,\bar y))(0)$ or ${\rm d}^2 \delta_{Y }(\bar y;{\rm d}_yf(\bar x,\bar y))(h)$ is finite. It follows that
$$-{\rm d}^2 \delta_{X \times Y}((\bar x,\bar y);{\rm d}f(\bar x,\bar y))(0,h)\leq -{\rm d}^2 \delta_{X}(\bar x;{\rm d}_x f(\bar x,\bar y))(u)-{\rm d}^2 \delta_{Y}(\bar y;{\rm d}_y f(\bar x,\bar y))(h).$$ Hence (\ref{2nddescentnew}) implies (\ref{2nddescent}). By Proposition \ref{local-to-lip}, the only proof left is to show that for any  $h\in T_Y(\bar y)$, we must have ${\rm d}_yf(\bar x,\bar y)(h)= 0$.  Since $\bar y$ is a local maximizer of $f(\bar x,y)$ over $Y$, we always have ${\rm d}_yf(\bar x,\bar y)(h)\leq 0$. Moreover,
we have
\begin{equation*}
{\rm d} f (\bar x, \bar y )(u,h)
= \lim\limits_{k\to \infty} \frac{f(x_k, y_k)-f(\bar x,\bar y)}{t_k}
\geq \lim\limits_{k\to \infty} \frac{f (x_k, \bar y )-f (\bar x, \bar y )}{t_k}
= {\rm d}_x f (\bar x, \bar y)(u).
\end{equation*}
Hence together with (\ref{d=dx+dy}) we must have ${\rm d}_yf(\bar x,\bar y)(h)= 0$ and hence the proof of the corollary is complete.
Particularly, by Proposition \ref{Calculation}, if $f$ is differentiable and $X$ is convex polyhedral, then ${\rm d}^2 \delta_{X }(\bar x;{\rm d}_xf(\bar x,\bar y))(0)=0$.
\end{proof}

To end this section, we summarize relationships between Nash equilibrium, local Nash equilibrium, calm local minimax points, local minimax points, and global minimax points.
Note that unlike in nonconvex optimization where global optima are always local optima (thus local optima always exist if global optima exist), without additional assumptions, a global minimax point may not be a local minimax point (and thus may not be a calm local minimax point) \cite[Example 3.4]{JC22}.

{\bf{Relations in the general case:}}
$$\mathcal{D}_{Nash} \subset \mathcal{D}_{localNash} \subset  \mathcal{D}_{calm-local} \subset \mathcal{D}_{local},\quad \mathcal{D}_{Nash} \subset \mathcal{D}_{global}.$$

{\bf{Relations for the smooth minimax problems under the weak sufficient condition for local optimality of the maximization problem} (see Corollary \ref{local-to-lip2}):}
$$\mathcal{D}_{Nash} \subset \mathcal{D}_{calm-local} = \mathcal{D}_{local}, \quad \mathcal{D}_{Nash} \subset \mathcal{D}_{global}.$$

{\bf{Relations for the smooth nonconvex-strongly-concave case (see Corollary \ref{corollary-strong-concave}):}}
$$\mathcal{D}_{Nash} \subset \mathcal{D}_{global} \subset \mathcal{D}_{calm-local} = \mathcal{D}_{local}.$$

Here, the notation ``$\subset$" entails that the inclusion can be strict. The notations $\mathcal{D}_{Nash}, \mathcal{D}_{localNash}, \mathcal{D}_{calm-local}, \mathcal{D}_{local}$, and $\mathcal{D}_{global}$ denote the set of Nash equilibrium, local Nash equilibrium, calm local minimax points, local minimax points, and global minimax points, respectively.

\section{Optimality conditions for the minimax problem}\label{nonsmooth-section}

In this section, we give first-order and second-order optimality conditions for the minimax problem (Min-Max).

\subsection{First-order optimality conditions for minimax problems}

First-order optimality conditions in the primal form are given as follows.

\begin{theorem}[first-order optimality conditions in primary form]\label{1st-non}
Let $(\bar x, \bar y) \in X \times  Y$.
\begin{itemize}
\item[(a)] Suppose that
   \begin{align}
  {\rm d}_xf(\bar x,\bar y)(u)>0 \quad &\forall\, u\in T_{X}(\bar x) \setminus \{0\}, \label{1stSuf-x-non} \\
  {\rm d}_y^+ f(\bar x,\bar y)(h)<0 \quad  &\forall\, h \in T_{ Y}(\bar y) \setminus \{0\}. \label{1stSuf-y-non-1}
   \end{align}
Then $(\bar x, \bar y)$ is a local Nash equilibrium and hence a calm local minimax point to problem \eqref{minimax}.
\item[(b)] Suppose that either $Y$ is the whole space or $f(\bar x, \cdot)$ is Lipschitz continuous around $\bar y$. Suppose further that $f(x,\cdot)$ is continuous for any $x \in X$ near $\bar x$.
If $(\bar x, \bar y)$ is a calm local minimax point to the minimax problem \eqref{minimax}, then
for any $u \in T_X(\bar{x})$, there exists $h \in T_Y(\bar{y})$ such that
\begin{eqnarray}\label{x-non}
&& {\rm d}^+f(\bar{x}, \bar{y})(u,h) \geq 0.
\end{eqnarray}
Besides,
\begin{equation}\label{y-non}
{\rm d}_y^+ f(\bar{x}, \bar{y})(h) \leq 0\quad \forall\, h \in T_Y(\bar{y}).
\end{equation}
\item[(c)] Suppose that either $Y$ is the whole space or $f(\bar x, \cdot)$ is Lipschitz continuous around $\bar y$. Moreover assume that $f$ is  semidifferentiable and the separation property holds for
the subderivative  at $(\bar x,\bar y)$.
If $(\bar x, \bar y)$ is a calm local minimax point to the minimax problem \eqref{minimax}, then
\begin{align}
  {\rm d}_x f(\bar x,\bar y)(u)\geq 0 \quad &\forall\, u\in T_{X}(\bar x), \label{1stSuf-x-non-new} \\
{\rm d}_y  f(\bar x,\bar y)(h)\leq 0 \quad  &\forall\, h \in T_{ Y}(\bar y) . \label{1stSuf-y-non-1-new}
   \end{align}
\end{itemize}
\end{theorem}
\begin{proof}

(a) By Proposition \ref{non-optimality} (ii) and the fact that ${\rm d}_y^+f(\bar x,\bar y)(h)=-{\rm d}_y(-f)(\bar x,\bar y)(h)$, we have that $\bar x$ is a local minimizer of $f(\cdot, \bar y)$ on $X$ and $\bar y$ is a local maximizer of $f(\bar x,\cdot)$ on $Y$. Thus, $(\bar x, \bar y)$ is a local Nash equilibrium to problem \eqref{minimax}.

(b) First, since $\bar y$ is a local maximum point of $f(\bar x,\cdot)$ on $Y$, \eqref{y-non} or equivalently, {${\rm d}_y(-f)(\bar x,\bar y)(h)\geq 0$} follows from
Proposition \ref{thm2} (i) (when $Y$ is the whole space) or Proposition \ref{non-optimality} (i) (when $f(\bar x, \cdot)$ is Lipschitz continuous around $\bar y$).
We now prove (\ref{x-non}). For this purpose we let $u \in T_X(\bar{x})$. Then there exist $t_k \downarrow 0, u_k \to u$ such that $x_k:= \bar x+t_ku_k \in X$. Take $\delta_k:=\|x_k-\bar x\|$. Then since $(\bar x,\bar y)$ is a calm local minimax point, there exist $\kappa >0$ and a sequence
$$y_k \in \argmax _{y \in Y\cap \mathbb{B}_{\tau(\delta_k)}(\bar{y})
	} f (x_k, y),$$
where $\tau(\delta)$ is the function defined in the definition of the calm local minimax point,
such that $\|y_k-\bar y\| \leq \kappa \|x_k-\bar x\|$ for sufficiently large $k$. Thus, by passing to a subsequence if necessary (without relabeling), there exists $h\in \mathbb{R}^m$ such that $h_k:=(y_k-\bar y)/t_k \to h$. By the definition of the contingent cone, we have $h\in T_{ Y}(\bar y)$.
Hence we have
\begin{eqnarray*}
\begin{aligned}
0 & \leq \lim\limits_{k \to \infty} \frac{f (x_k, y_k )-f(\bar{x}, \bar{y})}{t_k} \qquad \mbox{ by (\ref{localmm})}\\
& \leq \limsup\limits_{{t \downarrow 0} \atop {(u',h') \to (u,h)}} \frac{f (\bar{x}+tu', \bar{y}+th' )-f(\bar{x}, \bar{y})}{t} \\
& = {\rm d}^+ f(\bar x,\bar y)(u,h).
\end{aligned}
\end{eqnarray*}

(c) Since $f$ is semidifferentiable, ${\rm d}^+ f(\bar x,\bar y)(u,h)={\rm d} f(\bar x,\bar y)(u,h)$ and {${\rm d}_y^+ f(\bar{x}, \bar{y})(h)={\rm d}_y f(\bar{x}, \bar{y})(h)$}. Together with the separation property, the results follow from \eqref{x-non} and \eqref{y-non}.
\end{proof}

\begin{remark}

The first-order necessary condition (\ref{1stSuf-x-non-new}) is sharper than the one in Jiang and Chen \cite[(3.2a)-(3.2b)]{JC22} by the analysis in \cite[Remark 3.13]{JC22}. Theorem \ref{1st-non}(c) is stronger than Zhang et al. \cite[Theorem 3.12]{Zhang2022} even in the smooth case since the tangent cone is larger than the inner tangent cone. From the proof of Theorem \ref{1st-non}(a), we can see that (\ref{1stSuf-y-non-1}) is only used to show that $\bar y$ is a local maximizer of $f(\bar x,\cdot)$ on $Y$. Thus even in the smooth case, if we replace (\ref{1stSuf-y-non-1}) by that $\bar y$ is a local maximizer of $f(\bar x,\cdot)$ on $Y$ in Theorem \ref{1st-non}(a), then this sufficient optimality condition is weaker than Zhang et al.
 \cite[Theorem 3.14]{Zhang2022}.

\end{remark}

{Next, we give first-order necessary optimality conditions in the dual form for the minimax problem \eqref{minimax}.}

\begin{theorem}[first-order necessary optimality conditions in the dual form]
Let $(\bar x, \bar y)$ be a calm local minimax point to the minimax problem \eqref{minimax}. Suppose that $f$ is semidifferentiable and Lipschitz continuous around $(\bar x,\bar y)$ and the separation property holds for
the subderivative at $(\bar x,\bar y)$. Then
\begin{equation*}
0 \in \partial_x^\circ f(\bar{x}, \bar{y}) + \widehat{N}_X(\bar x),
\end{equation*}
\begin{equation*}
0 \in -\partial_y^\circ f(\bar{x}, \bar{y}) + \widehat{N}_Y(\bar y).
\end{equation*}
\end{theorem}
\begin{proof}
By Theorem \ref{1st-non} (c), since $(\bar x, \bar y)$ be a  calm local minimax point, we have
\begin{eqnarray*}
  {\rm d}_x f(\bar x,\bar y)(u)\geq 0 \quad \text { for all } u\in T_{X}(\bar x),
&& {\rm d}_y f(\bar{x}, \bar{y})(h) \leq 0 \text { for all } h \in T_Y(\bar{y}).
\end{eqnarray*}
Since
$f^\circ_x(\bar x,\bar y;u) \geq {\rm d}_x f(\bar x,\bar y)(u)$ and $(-f)^\circ_y(\bar x,\bar y;h)\geq -{\rm d}_y f(\bar{x}, \bar{y})(h) $,
it follows that
\begin{eqnarray*}
 f^\circ_x(\bar x,\bar y;u)\geq 0 \quad \text { for all } u\in T_{X}(\bar x),
&& (-f)^\circ_y(\bar x,\bar y;h) \geq 0 \text { for all } h \in T_Y(\bar{y}).
\end{eqnarray*}
Since for any $u\in \mathbb{R}^n, h\in \mathbb{R}^m$, we have (see e.g. \eqref{psicirc}),
\begin{eqnarray*}
 f^\circ_x(\bar x,\bar y;u)=\max_{\xi \in \partial_x^\circ f(\bar{x}, \bar{y})}\langle \xi, u\rangle,
&&(-f)^\circ_y(\bar x,\bar y;h)=\max_{\zeta \in \partial_y^\circ (-f)(\bar{x}, \bar{y})}\langle \zeta, h\rangle,
\end{eqnarray*} there exist $\xi \in \partial_x^\circ f(\bar{x}, \bar{y})$ and $\zeta \in \partial_y^\circ (-f)(\bar{x}, \bar{y})=-\partial_y^\circ f(\bar{x}, \bar{y})$ such that
\begin{eqnarray*}
\langle \xi, u\rangle \geq 0 \quad \forall u \in T_X(\bar x),&& \langle \zeta, h\rangle \geq 0 \quad \forall h \in T_Y(\bar y).
\end{eqnarray*}
Since the regular normal cone is the negative polar of the tangent cone (see e.g., \eqref{normalcone-equ}),  we have $-\xi\in \widehat{N}_X(\bar x)$, $-\zeta \in \widehat{N}_Y(\bar y)$.
\end{proof}

\subsection{Second-order optimality conditions for minimax problems}
In this section, we give  second-order optimality conditions for the minimax problem \eqref{minimax} for the general case.

\begin{theorem}[Second-order optimality conditions for the constrained minimax problem]\label{main}
Let $(\bar x, \bar y)\in X \times Y$. Suppose that $f$ is twice semidifferentiable at $(\bar x, \bar y)$, the separation property for the subderivative holds at $(\bar x,\bar y)$,  and the value ${\rm d}^2 \delta_X(\bar x;-{\rm d}_xf(\bar x,\bar y))(u)$ is finite for any $u\in T_{X}(\bar x) \cap \{u' | {\rm d}_x f(\bar x,\bar y)(u')  =0 \}$.

\begin{itemize}
\item[(a)]
Suppose that the first-order necessary optimality conditions
 \begin{align}
  {\rm d}_xf(\bar x,\bar y)(u)\geq 0 \quad &\text { for all } u\in T_{X}(\bar x), \label{firstordercond1}\\
{\rm d}_y f(\bar x,\bar y)(h)\leq 0 \quad  &\text { for all } h \in T_{ Y}(\bar y),\label{firstordercond2}
   \end{align}
and the second-order sufficient condition for  problem $\max_{y\in Y} f(\bar x, y)$ holds at $\bar y$, i.e., for any $h\in T_{ Y}(\bar y) \cap \{h'|{\rm d}_y f(\bar x,\bar y)(h')= 0 \} \setminus \{0\}$,
\begin{equation}\label{sufflower}
{\rm d}^2_{yy} f(\bar x,\bar y)(h) -{\rm d}^2 \delta_Y(\bar y;{\rm d}_yf(\bar x,\bar y))(h)<0.
\end{equation}
{ If $\delta_Y$ is twice epi-differentiable at $\bar y$ for ${\rm d}_y f(\bar x,\bar y)$ and   for any $u\in T_{X}(\bar x) \cap \{u' | {\rm d}_x f(\bar x,\bar y)(u') = 0 \} \setminus \{0\}$, there exists $h\in T_ Y({\bar y})$
such that
\begin{equation}\label{cxlower3-non-new}
{\rm d}^2 f(\bar x,\bar y)(u,h) +{\rm d}^2 \delta_X(\bar x;-{\rm d}_xf(\bar x,\bar y))(u)-{\rm d}^2 \delta_Y(\bar y;{\rm d}_yf(\bar x,\bar y))(h) >0,
\end{equation}}
then $(\bar x,\bar y)$ is a calm local minimax point to problem \eqref{minimax} and  the  following second-order growth condition holds: there exist $\delta_0>0$,   $\mu > 0$
 such that  for any $\delta \in (0,\delta_0]$, $x \in X \cap \mathbb{B}_{\delta}(\bar x)$, and $y \in Y\cap \mathbb{B}_{\delta}(\bar y)$, we have
\begin{equation*}
f(\bar x, y)+\varepsilon\|y-\bar y\|^2  \leq f(\bar x, \bar y) \leq \max _{y^{\prime}  \in  Y\cap \mathbb{B}_\delta(\bar y)} f(x, y^{\prime}) - \mu \|x-\bar x\|^2.
\end{equation*}
\item[(b)] Let $(\bar x,\bar y) \in X \times  Y$ be a calm local minimax point to problem \eqref{minimax}. Suppose that either $Y$ is the whole space or $f(\bar x,\cdot)$ is Lipschitz continuous around $\bar y$. Then the first-order necessary optimality conditions (\ref{firstordercond1})-(\ref{firstordercond2}) hold,  the second-order necessary condition for the maximum problem $\max_{y\in Y} f(\bar x, y)$ holds at $\bar y$, i.e., for any $h\in T_{ Y}(\bar y) \cap \{h'|{\rm d}_y f(\bar x,\bar y)(h') = 0\}$, we have
\begin{equation}\label{necelower}
{\rm d}^2_{yy} f(\bar x,\bar y)(h) -{\rm d}^2 \delta_Y(\bar y;{\rm d}_yf(\bar x,\bar y))(h) \leq 0,
\end{equation}
and for any $u \in T_X(\bar x) \cap \{u' | {\rm d}_x f(\bar x,\bar y)(u') = 0\} $, there exists $h\in T_ Y({\bar y})$
such that
 \begin{equation}\label{eqn4.11}
{\rm d}^2 f(\bar x,\bar y)(u,h) +{\rm d}^2 \delta_X(\bar x;-{\rm d}_xf(\bar x,\bar y))(u)-{\rm d}^2 \delta_Y(\bar y;{\rm d}_yf(\bar x,\bar y))(h) \geq 0.
\end{equation}
\end{itemize}
\end{theorem}
\begin{proof}
Note that if ${\rm d}_y f\left(\bar x, \bar y\right)(h) < 0$, by  (\ref{infinity}), ${\rm d}^2 \delta_Y(\bar y;{\rm d}_yf(\bar x,\bar y))(h) =\infty$. Hence the conditions {\eqref{cxlower3-non-new} and \eqref{eqn4.11}} have excluded this possibility. Combining with (\ref{firstordercond2}) we have
${\rm d}_y f(\bar x,\bar y)(h)=0 $ for all $ h \in T_{ Y}(\bar y).$

(a) Since $f$ is twice semidifferentiable at $(\bar x,\bar y)$, by (\ref{eqn2.2new}) and (\ref{twice-epi}), we have
$$-{\rm d}_y f(\bar x,\bar y)(h)={\rm d}_y (-f)(\bar x,\bar y)(h), \quad -{\rm d}_{yy}^2 f(\bar x,\bar y)(h)={\rm d}_{yy}^2 (-f)(\bar x,\bar y)(h).$$ By Proposition \ref{non-optimality2} (ii), the second-order sufficient condition for the maximum problem  implies that $\bar y$ is a local maximizer of $f(\bar x,\cdot)$ on $ Y$ and the second-order growth condition holds. Thus, there exist $\delta_0>0, \varepsilon>0$ such that for any $\delta \in (0,\delta_0], y\in Y$ satisfying $\|y-\bar{y}\| \leq \delta$, we have  $f(\bar x,y) +\varepsilon\|y-\bar y\|^2\leq f(\bar x,\bar y)$. Let $\tau(\delta):=\delta$, then $\tau(\delta) \rightarrow 0$ as $\delta \downarrow 0$.
Since
$ V_\delta(x):=V_\delta(x;\bar y):=\max_{y  \in  Y\cap \mathbb{B}_\delta(\bar y)}f(x,y),$
we have
\begin{equation}\label{eqn4.13} V_\delta( x)\geq f(x,y) \ \ \forall (x,y)\in \mathbb{B}_\delta(\bar x,\bar y)\cap (X\times Y), \qquad \mbox{ and } V_\delta(\bar x)=f(\bar x,\bar y).\end{equation}
We break the rest proof for (a) into two steps.

{\bf Step 1:} We show that for any fixed $u \in T_X(\bar x) \cap \{u' | {\rm d}_x f(\bar x,\bar y)(u') = 0 \}$ and $\delta \in (0,\delta_0]$,
\begin{equation}\label{lowerbd2}
	\begin{aligned}
		&{\rm d}^2 (V_{\delta}+\delta_X)(\bar x;0)(u) \\
		\geq& \sup_{{h \in T_{ Y}(\bar{y})}
		} \left\{ {\rm d}^2 f(\bar x,\bar y)(u,h) +{\rm d}^2 \delta_X(\bar x;-{\rm d}_xf(\bar x,\bar y)(u)-{\rm d}^2 \delta_Y(\bar y;{\rm d}_yf(\bar x,\bar y))(h)\right\}.
	\end{aligned}
\end{equation}
Since $u \in T_X(\bar x) \cap \{u' | {\rm d}_x f(\bar x,\bar y)(u') = 0\}$, by definition of the second subderivative in Definition \ref{Defn2.7}, there exist $t_k\downarrow 0$, $u_k\to u$ such that  $\bar x+t_ku_k \in X$ and
\begin{equation}\label{main-1}
{\rm d}^2 (V_{\delta}+\delta_X)(\bar x;0)(u)
= \lim\limits_{k \to \infty} \frac{V_{\delta}(\bar{x}+t_ku_k)-V_{\delta}(\bar{x})}{\frac{1}{2}t_k^2}.
\end{equation}
Since $\delta_{Y}$ is twice epi-differentiable at $\bar y$ for ${\rm d}_yf(\bar x,\bar y)$, by Definition \ref{Defn2.7} for any $h \in T_{ Y}(\bar{y})$, we can find a sequence $h_{k} \rightarrow h$ such that
$$-{\rm d}^2 \delta_Y(\bar y;{\rm d}_yf(\bar x,\bar y))(h)=-\lim\limits_{k \to \infty}\frac{\delta_Y(\bar y+t_kh_k)-t_k{\rm d}_yf(\bar{x}, \bar{y})(h_k)}{\frac{1}{2}t_k^2}.$$
If $y_k:=\bar{y}+t_kh _{k} \not\in  Y$ for each $k \in \mathbb{N}$, $-{\rm d}^2 \delta_Y(\bar y;{\rm d}_yf(\bar x,\bar y))(h) = - \infty$ but this is impossible by (\ref{cxlower3-non-new}). Hence we may assume that  $y_k:=\bar{y}+t_kh _{k} \in  Y$ for each $k \in \mathbb{N}$. Then
\begin{equation}\label{eqn4.16}-{\rm d}^2 \delta_Y(\bar y;{\rm d}_yf(\bar x,\bar y))(h)=\lim\limits_{k \to \infty}\frac{2{\rm d}_yf(\bar{x}, \bar{y})(h_k)}{t_k}.\end{equation}
Hence we have
\begin{eqnarray*}\label{}
\begin{aligned}
& \quad\ {\rm d}^2 (V_{\delta}+\delta_X)(\bar x;0)(u)  = \liminf\limits_{k \to \infty} \frac{V_{\delta}\left(\bar{x}+t_ku_k\right)-V_{\delta}(\bar{x})}{\frac{1}{2} t_k^2}  \quad \mbox{ by } \eqref{main-1}\\
&\geq \liminf\limits_{k \to \infty} \frac{f\left(\bar{x}+t_k u_k, \bar{y}+t_k h_k\right)-f(\bar{x}, \bar{y})}{\frac{1}{2} t_k^2}\quad \mbox{ by } \eqref{eqn4.13}\\
&= \liminf\limits_{k \to \infty} \left\{\frac{f\left(\bar{x}+t_k u_k, \bar{y}+t_k h_k\right)-f(\bar{x}, \bar{y})-t_k{\rm d}f(\bar{x}, \bar{y})(u_k,h_k)}{\frac{1}{2} t_k^2}+\frac{2{\rm d}_x f(\bar{x}, \bar{y})(u_k)}{t_k}+\frac{2{\rm d}_y f(\bar{x}, \bar{y})(h_k)}{t_k}\right\} \\
& \geq {\rm d}^2 f(\bar x,\bar y)(u,h) +{\rm d}^2 \delta_X(\bar x;-{\rm d}_xf(\bar x,\bar y))(u)-{\rm d}^2 \delta_Y(\bar y;{\rm d}_yf(\bar x,\bar y))(h),
\end{aligned}
\end{eqnarray*}
where the last inequality follows due to (\ref{eqn4.16}), that ${\rm d}f(\bar{x}, \bar{y})(u,h)={\rm d}_xf(\bar{x}, \bar{y})(u)+{\rm d}_yf(\bar{x}, \bar{y})(h)$ for any $(u,h) \in \mathbb{R}^n \times \mathbb{R}^m$, the assumption that for $ u\in T_{X}(\bar x) \cap \{u' | {\rm d}_x f(\bar x,\bar y)(u') = 0\}$, ${\rm d}^2 \delta_X(\bar x;-{\rm d}_xf(\bar x,\bar y))(u)$ is finite.
Thus (\ref{lowerbd2}) holds.

{\bf Step 2:} We show that for any $\delta \in (0,\delta_0]$ and $x \in X$ satisfying $\|x-\bar x\| \leq \delta$, we have
\begin{equation}\label{second-order-growth} \max _{y^{\prime}  \in  Y\cap \mathbb{B}_\delta(\bar y)} f\left(x, y^{\prime}\right) - f(\bar x,\bar y) \geq \beta \|x-\bar x\|^2\end{equation}
for some $\beta > 0$.

To the contrary, suppose that for some $\delta \in (0,\delta_0]$ and $x_k \in X$  with $\|x_k-\bar x\| \leq \delta$,
\begin{equation}\label{contradiction2}
\max _{y^{\prime}  \in  Y\cap \mathbb{B}_\delta(\bar y)} f\left(x_k, y^{\prime}\right)-f(\bar x,\bar y) \leq o(t_k^2),
\end{equation}
where $t_k:= \|x_k-\bar x\|$. Let $u_k:=(x_k-\bar x)/\|x_k-\bar x\|$, we have $t_k\downarrow 0$ and $\|u_k\|=1$. By passing to a subsequence if necessary, we may assume that $u_k \to u$ with $\|u\|=1$. We have $u \in T_{X}(\bar x) \setminus \{0\}$.
The assumed first-order condition gives us ${\rm d}_x f(\bar x,\bar y)(u) \geq 0$.
If ${\rm d}_x f(\bar x,\bar y)(u)>0$,
$$\max _{y^{\prime}  \in  Y\cap \mathbb{B}_\delta(\bar y)} f\left(x_k, y^{\prime}\right)-f(\bar x,\bar y) \geq f(x_k,\bar y) - f(\bar x,\bar y) \geq t_k{\rm d}_x f(\bar x,\bar y)(u)+o(t_k) > o(t_k) \geq o(t_k^2),$$
which is a contradiction to \eqref{contradiction2}.
If ${\rm d}_x f(\bar x,\bar y)(u) = 0$, by \eqref{lowerbd2} and (\ref{main-1}), we have
\begin{equation*}
\begin{aligned}
\max _{y^{\prime}  \in  Y\cap \mathbb{B}_\delta(\bar y)} f\left(x_k, y^{\prime}\right)-f(\bar x,\bar y) & = V_{\delta}(x_k) - V_{\delta}(\bar x)  \geq \frac{1}{2}t_k^2 \theta(\bar x,\bar y, u,h)+o(t_k^2),
\end{aligned}
\end{equation*}
where $$\theta(\bar x,\bar y, u,h):= \sup_{{h \in T_{ Y}(\bar{y})}
 } \left\{ {\rm d}^2 f(\bar x,\bar y)(u,h) +{\rm d}^2 \delta_X(\bar x;-{\rm d}_xf(\bar x,\bar y))(u)-{\rm d}^2 \delta_Y(\bar y;{\rm d}_yf(\bar x,\bar y)(h)\right\}.$$
It follows from \eqref{cxlower3-non-new} that $\theta(\bar x,\bar y, u,h)>0$. Hence we have a contradiction to \eqref{contradiction2} and consequently (\ref{second-order-growth}) holds.

Combining with the fact that $\bar y$ is a  maximizer of $f(\bar x,\cdot)$ on $ Y\cap \mathbb{B}_\delta(\bar y)$, it follows that
$$f(\bar x, y) \leq f(\bar x,\bar y)\leq \max _{y^{\prime}  \in  Y\cap \mathbb{B}_\delta(\bar y)} f\left(x, y^{\prime}\right) \quad \forall (x,y) \in \mathbb{B}_\delta(\bar x,\bar y) \cap (X\times Y).$$
Thus, $(\bar x, \bar y)$ is a calm local minimax point to problem \eqref{minimax}.

(b) First, by Theorem \ref{1st-non} (c) and Proposition \ref{non-optimality2} (i), we have the first- and second-order conditions for the maximization problem.

Second, since $(\bar x,\bar y)$ is a calm local minimax point to problem \eqref{minimax}, by Lemma
\ref{locamm-equi}, there exist a $\delta_0 >0$ and a function $\tau: \mathbb{R}_{+} \rightarrow \mathbb{R}_{+}$ which is calm at $0$ satisfying $\tau(\delta) \rightarrow 0$ as $\delta \downarrow 0$, such that for any $\delta \in\left(0, \delta_0\right]$ and any $x \in X\cap \mathbb{B}_\delta (\bar x)$, we have
\begin{equation}
 f(\bar{x}, \bar{y}) \leq \max _{y^{\prime}  \in  Y\cap \mathbb{B}_{\tau(\|x-\bar{x}\|)}(\bar y)} f\left(x, y^{\prime}\right) .\label{eqn4.21}
\end{equation}
For any $u\in T_{X}(\bar x) \cap \{u' | {\rm d}_x f(\bar x,\bar y)(u') = 0\} $, there exist $t_k\downarrow 0$, $u_k\to u$ such that $x_k:=\bar x+t_ku_k \in X$ and
$${\rm d}^2 \delta_X(\bar x;-{\rm d}_xf(\bar x,\bar y))(u) = \lim\limits_{k\to \infty} \frac{2 {\rm d}_xf(\bar{x}, \bar{y})(u_k)}{t_k}.$$
Since $\tau $ is calm,  there exist $\kappa>0$ and
\begin{equation*}
y_k \in \argmax _{y^{\prime}  \in  Y\cap \mathbb{B}_{\tau(\|x_k-\bar x\|)}(\bar y)} f\left(x_k, y^{\prime}\right),
\end{equation*}
 such that $\|y_k-\bar y\| \leq \kappa \|x_k-\bar x\|$. Thus, by passing to a subsequence if necessary (without relabeling), there exists $h\in \mathbb{R}^m$ such that $h_k:=(y_k-\bar y)/t_k \to h$. By the definition of the contingent cone, we have $h\in T_{ Y}(\bar y)$.

Thus,
\begin{eqnarray*}\label{}
\begin{aligned}
0 & \leq \limsup\limits_{k \to \infty} \frac{f\left(\bar{x}+t_k u_k, \bar{y}+t_kh_k\right)-f(\bar{x}, \bar{y})}{\frac{1}{2}t_k^2}\quad \mbox{ by } \eqref{eqn4.21} \\
& =\limsup\limits_{k \to \infty} \frac{f\left(\bar{x}+t_k u_k, \bar{y}+t_kh_k\right)-f(\bar{x}, \bar{y})-t_k{\rm d}f(\bar{x}, \bar{y})(u_k,h_k)+t_k{\rm d}_xf(\bar{x}, \bar{y})(u_k)+t_k{\rm d}_yf(\bar{x}, \bar{y})(h_k)}{\frac{1}{2}t_k^2} \\
& = {\rm d}^2 f(\bar x,\bar y)(u,h) +{\rm d}^2 \delta_X(\bar x;-{\rm d}_xf(\bar x,\bar y))(u) + \limsup\limits_{k \to \infty} \frac{t_k{\rm d}_yf(\bar{x}, \bar{y})(h_k)}{\frac{1}{2}t_k^2} \\
& \leq {\rm d}^2 f(\bar x,\bar y)(u,h) +{\rm d}^2 \delta_X(\bar x;-{\rm d}_xf(\bar x,\bar y))(u)+ \limsup\limits_{{t \downarrow 0,h' \to h} \atop {\bar y+th'\in Y}} \frac{t{\rm d}_yf(\bar{x}, \bar{y})(h')}{\frac{1}{2}t^2} \\
& = {\rm d}^2 f(\bar x,\bar y)(u,h) +{\rm d}^2 \delta_X(\bar x;-{\rm d}_xf(\bar x,\bar y))(u)- \liminf\limits_{{t \downarrow 0,h' \to h} \atop {\bar y+th'\in Y}} \frac{-t{\rm d}_yf(\bar{x}, \bar{y})(h')}{\frac{1}{2}t^2} \\
& = {\rm d}^2 f(\bar x,\bar y)(u,h) +{\rm d}^2 \delta_X(\bar x;-{\rm d}_xf(\bar x,\bar y))(u)-{\rm d}^2 \delta_Y(\bar y;{\rm d}_yf(\bar x,\bar y))(h),
\end{aligned}
\end{eqnarray*} where the second equality follows from the assumption for the separation property.
\end{proof}

When the separation property holds for the second subderivative, our second-order necessary optimality conditions are reduced to the following { simpler but relaxed forms.}
\begin{corollary}\label{cor4.1}
Let $(\bar x,\bar y) \in X \times  Y$ be a calm local minimax point to problem \eqref{minimax}. Suppose that $f$ is twice semidifferentiable at $(\bar x, \bar y)$, the separation property holds for the second subderivative of $f$ at $(\bar x,\bar y)$, and the value ${\rm d}^2 \delta_X(\bar x;-{\rm d}_xf(\bar x,\bar y))(u)$ is finite for any $u\in T_{X}(\bar x) \cap \{u' | {\rm d}_x f(\bar x,\bar y)(u')  =0 \}$. Then, for any $u \in T_X(\bar x) \cap \{u' | {\rm d}_x f(\bar x,\bar y)(u') = 0\} $, there exists $h\in T_ Y({\bar y})$
such that
 \begin{equation}\label{eqn4.23}
{\rm d}^2_{xx} f(\bar x,\bar y)(u) +2{\rm d}^2_{xy} f(\bar x,\bar y)(u,h)  +{\rm d}^2 \delta_X(\bar x;-{\rm d}_xf(\bar x,\bar y))(u) \geq 0,
\end{equation}
and for any $h\in T_{ Y}(\bar y) \cap \{h'|{\rm d}_y f(\bar x,\bar y)^T(h') = 0\}$, we have
$${\rm d}^2_{yy} f(\bar x,\bar y)(h) -{\rm d}^2 \delta_Y(\bar y;{\rm d}_yf(\bar x,\bar y))(h) \leq 0.$$
\end{corollary}
\begin{proof}
By the analysis at the beginning of the proof of Theorem \ref{main}, we have
${\rm d}_y f(\bar x,\bar y)(h)=0 $ for all $ h \in T_{ Y}(\bar y).$ Then, \eqref{necelower} together with \eqref{eqn4.11} and \eqref{2ndseparation} give us \eqref{eqn4.23}.
\end{proof}
By Proposition \ref{non-optimality2},   the condition ${\rm d}^2_{xx} f(\bar x,\bar y)(u) +{\rm d}^2 \delta_X(\bar x;-{\rm d}_xf(\bar x,\bar y))(u) \geq 0 \ \forall u\in T_{X}(\bar x) \cap \{u' | \nabla_x f(\bar x,\bar y)^Tu'  =0 \}$ is  necessary for  $\bar x$ to be a local minimizer  for the  problem $\min_{x \in X} f(x,\bar y)$. But this corresponds to the concept of local Nash equilibrium. Since a local Nash equilibrium must be a calm local minimax point but not vice versa,  the term ${\rm d}^2_{xy} f(\bar x,\bar y)(u,h)$ can not in general be dismissed in Corollary \ref{cor4.1}.

Jiang and Chen gave necessary optimality conditions for local minimax points in \cite[Theorems 3.11 and 3.17]{JC22}, our approach differs in the following ways:
\begin{itemize}
\item[(i)] We do not require the convexity of sets $X$ and $Y$. Additionally, by utilizing second subderivatives of indicator functions, we characterize optimality in all critical directions, namely $T_{X}(\bar x) \cap \{u' | {\rm d}_x f(\bar x,\bar y)(u') =0 \}$ and $T_{ Y}(\bar y) \cap \{h'|{\rm d}_y f(\bar x,\bar y)(h') = 0\}$. However, \cite[Theorems 3.11 and 3.17]{JC22} only focus on a subset of critical directions.
\item[(ii)] Our optimality conditions effectively capture the nested structure of the minimax problem. Specifically, we incorporate second-order information on $y$ in \eqref{eqn4.23} (or \eqref{eqn4.11}) by leveraging the property of calm local minimax points. We will provide an example in the last section to demonstrate that even when calm local minimax points are identical to local minimax points, our optimality conditions can be sharper than those in \cite{JC22}.
\end{itemize}

\section{Special cases and comparisons with existing related works}\label{section-compare}

In this section, we derive optimality conditions for some special cases when the functions have more  properties and the  constraint sets have some specific structures. To derive the corresponding optimality conditions for other cases, one can use the calculation rules presented in Proposition \ref{Calculation} in conjunction with the method we use in this section.  We also compare our results with existing works and demonstrate that our optimality conditions for local optimality can be more appropriate for some minimax problems.

{\subsection{Set-constrained systems}
In this section we consider the minimax problem
\begin{equation}\label{set-constrained}
  \min_{x\in X}\max_{y\in  Y} f(x,y),
  \end{equation}where  $f$ is twice continuously differentiable and the  constraints are defined by
\begin{eqnarray*}
X:= \{ x\in \mathbb{R}^n | \phi(x)\in C\},\quad
Y:= \{ y\in \mathbb{R}^m | \varphi (y) \in D\},
\end{eqnarray*}
where  $\phi:\mathbb{R}^n \rightarrow \mathbb{R}^{p}$ and $\varphi:\mathbb{R}^m  \to \mathbb{R}^q$ are twice continuously differentiable, $C\subseteq \mathbb{R}^{p}$ and $D\subseteq \mathbb{R}^{q}$ are convex.

Define the critical cones for the minimization and the maximization problem at $(\bar x, \bar y)$ respectively:
$$C_{\min}(\bar x,\bar y):= \{u\in \mathbb{R}^n| \nabla_x f\left(\bar x, \bar y\right)u=   0, \nabla \phi(\bar x)u \in T_{C}(\phi(\bar x))\},$$
$$C_{\max}(\bar x,\bar y):= \{h\in \mathbb{R}^m| \nabla_y f\left(\bar x, \bar y\right)h=  0, \nabla \varphi(\bar y)h \in T_{D}(\varphi(\bar y))\}.$$

Define the set of multipliers corresponding to the minimization and the maximization problem at $(\bar x, \bar y)$ respectively:
\begin{eqnarray*}
\Lambda_{\min}(\bar x,\bar y)&:=& \left \{ \alpha \in N_C(\phi(\bar x))\mid \nabla_x f(\bar x,\bar y)+\nabla \phi(\bar x)^T\alpha=0 \right \},\\
\Lambda_{\max}(\bar x,\bar y)&:=& \left \{ \beta \in N_D(\varphi(\bar y))\mid -\nabla_y f(\bar x,\bar y)+ \nabla \varphi(\bar y)^T\beta=0\right \}.
\end{eqnarray*}

Denote the Lagrangian functions to the minimization and the maximization problem by
\begin{eqnarray*}
L_{\min}(x,y,\alpha,\beta) := f(x,y) + \phi(x)^T\alpha - \varphi(y)^T\beta, \quad L_{\max}(y,\beta;x) := f(x,y) - \varphi(y)^T\beta .
\end{eqnarray*}

Now, we can give  second-order optimality conditions for the minimax problem  (\ref{set-constrained}).

  \begin{theorem}\label{second-dual}Let $(\bar x,\bar y) \in X \times  Y$.  Suppose that the MSCQ holds  for the system $\phi(x)\in C$ at $\bar x$ and $\varphi(y)\in D$ at $\bar y$, respectively. Suppose $C$ is  parabolically derivable at $\phi(\bar x)$ for all vectors $\nabla \phi(\bar x) u$ where $u\in C_{\min}(\bar x,\bar y)$,  parabolically regular at $\phi(\bar x)$ for every $
  \alpha \in \Lambda_{\min}(\bar x,\bar y)$; $D$ is  parabolically derivable at $\varphi(\bar y)$ for all vectors $\nabla \varphi(\bar y) h$ where $h\in C_{\max}(\bar x,\bar y)$,  parabolically regular at $\varphi(\bar y)$ for every $
  \beta \in \Lambda_{\max}(\bar x,\bar y)$ {(e.g., when $C$ and $D$ are convex polyhedral sets, or second-order cones, or cones of positive semidefinite symmetric matrices)}.
  \begin{itemize}
 \item[(a)] Suppose
 the second-order sufficient optimality condition for the maximization holds, i.e.,
For any $h\in C_{\max}(\bar x,\bar y) \setminus \{0\}$, there exists a multiplier $\beta \in \Lambda_{\max}(\bar x,\bar y) $ such that
\begin{equation*}
h^T \nabla^2_{yy}L_{\max}(\bar y,\beta; \bar x) h + \sigma_{T_{D}^2\left(\varphi(\bar y);\nabla \varphi(\bar y)h\right)}(\beta)< 0.
\end{equation*}
If for any $u\in C_{\min}(\bar x,\bar y) \setminus \{0\}$, there exist $h\in C_{\max}(\bar x,\bar y)$  and a multiplier $\alpha \in \Lambda_{\min}(\bar x,\bar y)$ such that for any $\beta \in \Lambda_{\max}(\bar x,\bar y) $,
\begin{equation*}
\nabla^2_{(x,y)} L_{\min}(\bar x,\bar y,\alpha,\beta)((u,h),(u,h)) - \sigma_{T_{C}^2\left(\phi(\bar x);\nabla \phi(\bar x)u\right)}(\alpha)+ \sigma_{T_{D}^2\left(\varphi(\bar y);\nabla \varphi(\bar y)h\right)}(\beta)> 0,
\end{equation*}
then $(\bar x,\bar y)$ is a calm local minimax point to problem \eqref{set-constrained} {with the second-order growth condition.
}

  \item[(b)] Suppose that  $(\bar x,\bar y)$ is a calm local minimax point to problem \eqref{set-constrained}. Then the following second-order necessary optimality conditions for the maximization hold:
 for any $h \in C_{\max}(\bar x,\bar y)$, there exists a multiplier $\beta \in \Lambda_{\max}(\bar x,\bar y)$ such that
\begin{equation*}
h^T \nabla^2_{yy}L_{\max}(\bar y,\beta; \bar x) h + \sigma_{T_{D}^2\left(\varphi(\bar y);\nabla \varphi(\bar y)h\right)}(\nabla_y f(\bar x,\bar y)) \leq 0,
\end{equation*}
and for any $u \in C_{\min}(\bar x,\bar y)$, there exist $h \in C_{\max}(\bar x,\bar y)$ and a multiplier $\alpha \in \Lambda_{\min}(\bar x,\bar y)$ such that for any $\beta \in \Lambda_{\max}(\bar x,\bar y) $,
\begin{equation*}
\nabla^2_{(x,y)} L_{\min}(\bar x,\bar y,\alpha,\beta)((u,h),(u,h)) - \sigma_{T_{C}^2\left(\phi(\bar x);\nabla \phi(\bar x)u\right)}(\alpha)+ \sigma_{T_{D}^2\left(\varphi(\bar y);\nabla \varphi(\bar y)h\right)}(\beta) \geq 0.
\end{equation*}
\end{itemize}
   \end{theorem}}

\begin{proof}

 (i) Since $f$ is twice continuously differentiable, it is obvious that $f$ is twice semidifferentiable, Lipschitz continuous, and the separation property holds at any points.

{ (ii) We show that the value ${\rm d}^2 \delta_X(\bar x;-\nabla_xf(\bar x,\bar y))(u)$ is finite for any $u\in T_{X}(\bar x) \cap\{ \nabla_x f(\bar x,\bar y)\}^\perp$. Besides, the nonemptyness of the multiplier sets $\Lambda_{\min}(\bar x,\bar y)$ and $\Lambda_{\max}(\bar x,\bar y)$  is equivalent to the first-order conditions \eqref{firstordercond1} and \eqref{firstordercond2}.

By Proposition \ref{system}, $N_{X}^{p}(\bar x)=\widehat{N}_{X}(\bar x)=N_{X}(\bar x)$ and $N_{Y}^{p}(\bar y)=\widehat{N}_{Y}(\bar y)=N_{Y}(\bar y)$. Since $\alpha \in \Lambda_{\min}(\bar x,\bar y)$ and $\beta \in \Lambda_{\max}(\bar x,\bar y)$, we have that $-\nabla_x f(\bar x,\bar y)=\nabla \phi(\bar x)\alpha \in \nabla \phi(\bar x)N_C(\phi(\bar x))$ and $\nabla_y f(\bar x,\bar y)=\nabla \varphi(\bar y)\beta \in \nabla \varphi(\bar y)N_D(\varphi(\bar y))$. With \cite[Proposition 4.2]{ABM21}, this is equivalent to saying that
$$-\nabla_x f(\bar x,\bar y) \in N_{X}(\bar x), \quad \nabla_y f(\bar x,\bar y) \in N_{Y}(\bar y),$$
 which, by \eqref{normalcone-equ}, is equivalent to saying that
\begin{align*}
\nabla_xf(\bar x,\bar y)^Tu\geq 0 \quad &\text { for all } u\in T_{X}(\bar x), \\
\nabla_y f(\bar x,\bar y)^Th\leq 0 \quad  &\text { for all } h \in T_{ Y}(\bar y).
   \end{align*}
By Proposition \ref{Calculation}, the value ${\rm d}^2 \delta_X(\bar x;-\nabla_xf(\bar x,\bar y))(u)$ is finite for any $u\in T_{X}(\bar x) \cap \{ \nabla_x f(\bar x,\bar y)\}^\perp$.

(iii) We show that $\delta_Y$ is twice epi-differentiable at $\bar y$ for $\nabla_y f(\bar x,\bar y)$. By Proposition \ref{system} and the MSCQ assumption, $Y$ is parabolically regular at $\bar y$ for $ \nabla_yf(\bar x,\bar y) \in N_{Y}(\bar y)$ and is parabolically derivable at $\bar y$ for all vectors in $T_Y(\bar y) \cap \{\nabla_yf(\bar x,\bar y)\}^\perp$. Applying Proposition \ref{Tm3.6ABM},
$\delta_Y$ is twice epi-differentiable at $\bar y$ for $\nabla_y f(\bar x,\bar y)$. }

(iv) Using Proposition \ref{Calculation} (iii), {we have
$$
{\rm d}^2 \delta_X(\bar x;-\nabla_xf(\bar x,\bar y))(u)= \max_{\alpha \in \Lambda_{\min}(\bar x,\bar y)} \{ \langle \alpha, \nabla^2 \phi(\bar x)(u,u) \rangle -\sigma_{T_{C}^{2}(\phi(\bar x);\nabla \phi(\bar x)u)}(\alpha)\},
$$
\begin{equation*}
\begin{aligned}
-{\rm d}^2 \delta_Y(\bar y;\nabla_yf(\bar x,\bar y))(h) & = & -\max_{\beta \in \Lambda_{\max}(\bar x,\bar y)} \{ \langle \beta, \nabla^2 \varphi(\bar y)(h,h) \rangle -\sigma_{T_{D}^{2}(\varphi(\bar y);\nabla \varphi(\bar y)h)}(\beta)\} \\
& = &\min_{\beta \in \Lambda_{\max}(\bar x,\bar y)} \{- \langle \beta, \nabla^2 \varphi(\bar y)(h,h) \rangle +\sigma_{T_{D}^{2}(\varphi(\bar y);\nabla \varphi(\bar y)h)}(\beta)\}.
\end{aligned}
\end{equation*}
Thus,
\begin{equation*}
\begin{aligned}
&\quad\ \nabla^2 f(\bar x,\bar y)((u,h),(u,h)) +{\rm d}^2 \delta_X(\bar x;-\nabla_xf(\bar x,\bar y))(u)-{\rm d}^2 \delta_Y(\bar y;\nabla_yf(\bar x,\bar y))(h) \\
& =\nabla^2 f(\bar x,\bar y)((u,h),(u,h)) +\max_{\alpha \in \Lambda_{\min}(\bar x,\bar y)} \{ \langle \alpha, \nabla^2 \phi(\bar x)(u,u) \rangle -\sigma_{T_{C}^{2}(\phi(\bar x);\nabla \phi(\bar x)u)}(\alpha)\} \\
& \quad+ \min_{\beta \in \Lambda_{\max}(\bar x,\bar y)} \{- \langle \beta, \nabla^2 \varphi(\bar y)(h,h) \rangle +\sigma_{T_{D}^{2}(\varphi(\bar y);\nabla \varphi(\bar y)h)}(\beta)\}.
\end{aligned}
\end{equation*}}
With Theorem \ref{main}, we have the desired optimality conditions.
\end{proof}

Next, we consider the special case where the constraint sets $X$ and $Y$ involving only equalities and inequalities. Let $C=\mathbb{R}^{p_1}_- \times \{0\}^{p_2}$ and $D=\mathbb{R}^{q_1}_- \times \{0\}^{q_2}$. Then we have
 $$C_{\min}(\bar x,\bar y)=\left\{u \in \mathbb{R}^n \mid \nabla \phi_i(\bar x) u \leq 0, i \in I_\phi(\bar x) ; \nabla \phi_j(\bar x) u=0, j=1,...,p_2 ; \nabla_x f\left(\bar x,\bar y\right) u = 0\right\},$$
  $$C_{\max}(\bar x,\bar y)=\left\{h \in \mathbb{R}^m \mid \nabla \varphi_i(\bar x) h \leq 0, i \in I(\bar y) ; \nabla \varphi_j(\bar y) h=0, j=1,...,q_2 ; \nabla_y f\left(\bar x,\bar y\right) h = 0\right\},$$
where $I_\phi(\bar x):=\{i=1,...,p_1 \mid \phi_i(\bar x)=0\}$ and $I_\varphi(\bar y):=\{i=1,...,q_1\mid\varphi_i(\bar y)=0\}$,
\begin{eqnarray*}
\Lambda_{\min}(\bar x,\bar y)&=& \left \{ \alpha:=(\alpha_1,\alpha_2) \in \mathbb{R}^{p_1}_+ \times \mathbb{R}^{p_2} \mid \nabla_x f(\bar x,\bar y)+\nabla \phi(\bar x)^T\alpha=0,  \alpha_1 \perp \phi_{\leq}(\bar x) \right \},\\
\Lambda_{\max}(\bar x,\bar y)&=& \left \{ \beta:=(\beta_1,\beta_2) \in \mathbb{R}^{q_1}_+ \times \mathbb{R}^{q_2} \mid -\nabla_y f(\bar x,\bar y)+ \nabla \varphi(\bar y)^T\beta=0,  \beta_1 \perp \varphi_{\leq}(\bar y) \right \},
\end{eqnarray*}
where $\phi_{\leq}(\bar x):=(\phi_1(\bar x), ... ,\phi_{p_1}(\bar x))^T$ and $\varphi_{\leq}(\bar y):=(\varphi_1(\bar y), ... ,\varphi_{q_1}(\bar y))^T$.
 \begin{theorem}[inequalities and equalities systems]\label{cor5.1}
 Let  $C=\mathbb{R}^{p_1}_- \times \{0\}^{p_2}$, $D:=\mathbb{R}^{q_1}_- \times \{0\}^{q_2}$ and $(\bar x,\bar y) \in X \times  Y$. Suppose that the MSCQ holds for the system $\phi(x)\in C$ at $\bar x$ and $\varphi(y)\in D$ at $\bar y$, respectively (e.g., when the MFCQ holds or the functions $\phi(x)$ and $\varphi(y)$ are linear).
  \begin{itemize}
  \item[(a)] Suppose that for any $h\in C_{\max}(\bar x,\bar y) \setminus \{0\}$, there exists a multiplier $\beta \in \Lambda_{\max}(\bar x,\bar y) $ such that
\begin{equation*}
h^T \nabla^2_{yy}L_{\max}(\bar y,\beta; \bar x) h < 0,
\end{equation*}
and for any $u\in C_{\min}(\bar x,\bar y) \setminus \{0\}$, there exist $h\in C_{\max}(\bar x,\bar y)$  and a multiplier $\alpha \in \Lambda_{\min}(\bar x,\bar y)$ such that for any $\beta \in \Lambda_{\max}(\bar x,\bar y) $,
\begin{equation}\label{eqn5.2}
\nabla^2_{(x,y)} L_{\min}(\bar x,\bar y,\alpha,\beta)((u,h),(u,h))> 0.
\end{equation}
Then, $(\bar x,\bar y)$ is a calm local minimax point to the problem \eqref{set-constrained}
 {with the second-order growth condition.
}

 \item[(b)] Suppose that $(\bar x,\bar y)$ is a calm local minimax point to problem \eqref{set-constrained}.
Then for any $h \in C_{\max}(\bar x,\bar y)$, there exists a multiplier $\beta \in \Lambda_{\max}(\bar x,\bar y)$ such that
\begin{equation}\label{eqn5.5}
h^T \nabla^2_{yy} L_{\max}(\bar y,\beta; \bar x) h \leq 0,
\end{equation} and for any $u \in C_{\min}(\bar x,\bar y)$, there exist $h \in C_{\max}(\bar x,\bar y)$ and a multiplier $\alpha \in \Lambda_{\min}(\bar x,\bar y)$ such that for any $\beta \in \Lambda_{\max}(\bar x,\bar y) $,
\begin{equation}\label{eqn5.4}
\nabla^2_{(x,y)} L_{\min}(\bar x,\bar y,\alpha,\beta)((u,h),(u,h)) \geq 0.
\end{equation}
\end{itemize}
   \end{theorem}
\begin{proof}
Since $C$ and $D$ are convex polyhedral, the result follows from Proposition \ref{Calculation} (ii) and Theorem \ref{second-dual}.
\end{proof}

To compare our results with the existing literature, which primarily focused on local minimax points, we first establish the following sufficient condition under which  calm local minimax points coincide with local minimax points.

\begin{lemma}\label{lemma5.1}
Let $\left(\bar x,\bar y\right)$ be a local minimax point of problem \eqref{set-constrained} with  $C=\mathbb{R}^{p_1}_- \times \{0\}^{p_2}$ and $D=\mathbb{R}^{q_1}_- \times \{0\}^{q_2}$. Suppose that the MSCQ holds for the system $\phi(x)\in C$ at $\bar x$ and $\varphi(y)\in D$ at $\bar y$, respectively. Moreover, assume that the weak sufficient condition holds, i.e., for any $h\in C_{\max}(\bar x,\bar y) \setminus \{0\}$, there exists a multiplier $\beta \in \Lambda_{\max}(\bar x,\bar y) $ such that
\begin{equation*}
h^T \nabla^2_{yy}L_{\max}(\bar y,\beta; \bar x) h < 0.
\end{equation*}
Then $(\bar x, \bar y)$ is a calm local minimax point.
\end{lemma}
\begin{proof}
By Proposition \ref{Calculation} (ii) and (iii),
$
{\rm d}^2 \delta_{X }(\bar x;\nabla_xf(\bar x,\bar y))(0)=0$ and
$$-{\rm d}^2 \delta_{Y }(\bar y;\nabla_yf(\bar x,\bar y))(h)=-\max_{\beta \in \Lambda_{\max}(\bar x,\bar y)} \langle \beta, \nabla^2 \varphi(\bar y)(h,h) \rangle = \min_{\beta \in \Lambda_{\max}(\bar x,\bar y)} \langle \beta, -\nabla^2 \varphi(\bar y)(h,h) \rangle.$$ Thus, the condition \eqref{2nddescentnew} in Corollary \ref{local-to-lip2} holds and we have the desired result.
\end{proof}

We now compare our results with the one obtained by
Dai and Zhang in \cite[Theorems 3.2 and  3.1]{DaiZh-2020} for the case where the constraints for the maximization problem is independent of $x$. From the proof of Corollary \ref{Thm3.1D}, we can see that \eqref{eqn5.2} and   \eqref{eqn5.4} implies \eqref{eqn5.7} and \eqref{eqn5.6} respectively and hence our sufficient and necessary conditions in Theorem \ref{cor5.1}(a) and Theorem \ref{cor5.1}(b) are sharper than the one in  Corollary \ref{Thm3.1D}(i) and  Corollary \ref{Thm3.1D}(ii) respectively.  Note that Corollary \ref{Thm3.1D} obtains the same sufficient and necessary optimality conditions as the one given  in \cite[Theorems 3.2 and 3.1]{DaiZh-2020} under much weaker assumptions. In particular we do not need to assume that the Jacobian uniqueness condition holds.

\begin{corollary}\label{Thm3.1D}   Let  $C=\mathbb{R}^{p_1}_- \times \{0\}^{p_2}$, $D=\mathbb{R}^{q_1}_- \times \{0\}^{q_2}$ and $(\bar x,\bar y) \in X \times  Y$. Suppose that the MSCQ holds for the system $\phi(x)\in C$ at $\bar x$ and $\varphi(y)\in D$ at $\bar y$, respectively (e.g., when the MFCQ holds or the functions $\phi(x)$ and $\varphi(y)$ are linear).
\begin{itemize}
\item[(i)]Suppose that there exists a multiplier $\bar \beta \in \Lambda_{\max}(\bar x,\bar y)$ such that $\nabla_{yy}^2 L_{\max}(\bar y,\bar \beta;\bar x) \prec 0$. Suppose further that for any $u\in C_{\min}(\bar x,\bar y) \setminus \{0\}$, there exist $h\in C_{\max}(\bar x,\bar y)$ and a multiplier $\alpha \in \Lambda_{\min}(\bar x,\bar y)$ such that for any $\beta \in \Lambda_{\max}(\bar x,\bar y)$,
\begin{equation}\label{eqn5.7}
\nabla^2_{(x,y)} L_{\min}(\bar x,\bar y,\alpha,\beta)((u,h),(u,h))> 0.
\end{equation}
Then, $(\bar x,\bar y)$ is a local minimax point to problem \eqref{set-constrained} with the second-order growth condition.
\item[(ii)] Let $\left(\bar x,\bar y\right)$ be a local minimax point of problem \eqref{set-constrained}. Assume that there exists a multiplier $\beta \in \Lambda_{\max}(\bar x,\bar y)$ such that the matrix $\nabla_{yy}^2 L_{\max}(\bar y,\beta;\bar x)$ is nonsingular. Then for any $u \in C_{\min}(\bar x,\bar y)$, there exists a multiplier $\alpha \in \Lambda_{\min}(\bar x,\bar y)$ such that
\begin{equation}\label{eqn5.6}
u^T \left(\sum_{i=1}^{p} \alpha_i \nabla^2 \phi_i(\bar x) +[\nabla_{xx}^2 L_{\max}-\nabla_{xy}^2 L_{\max}(\nabla_{yy}^2 L_{\max})^{-1}\nabla_{yx}^2 L_{\max}](\bar y,\beta;\bar x) \right)u \geq 0.
\end{equation}
\end{itemize}
\end{corollary}
\begin{proof}
By Lemma \ref{lemma5.1}, the local minimax is equivalent to the calm local minimax under the assumption in (i) or (ii).

(i) Since $\nabla_{yy}^2 L_{\max}(\bar y,\bar \beta;\bar x) \prec 0$,  Theorem \ref{cor5.1} (a)(i) holds and we have the desired result.

(ii) Since $L_{\min}(\bar x,\bar y,\alpha,\beta)=L_{\max}(\bar y,\beta;\bar x) + \phi(\bar x)^T \alpha$, we have that $\nabla_{yy}^2 L_{\min}(\bar x,\bar y,\alpha,\beta)=\nabla_{yy}^2 L_{\max}(\bar y,\beta;\bar x)$ is nonsingular. For each $u\in C_{\min}(\bar x,\bar y) $, let $$h^*:=-\nabla_{yy}^2 L_{\min}(\bar x,\bar y,\alpha,\beta)^{-1}\nabla_{xy}^2 L_{\min}(\bar x,\bar y,\alpha,\beta)^T u.$$ Then we have
\begin{eqnarray*}\label{}
\begin{aligned}
& \quad\ \sup_{h \in C_{\max}(\bar x,\bar y)} \nabla^2_{(x,y)}L_{\min}(\bar x,\bar y,\alpha,\beta)((u,h),(u,h))  \\
& \leq \sup_{h\in \mathbb{R}^m} \nabla^2_{(x,y)} L_{\min}(\bar x,\bar y,\alpha,\beta)((u,h),(u,h))  \\
& = \sup_{h\in \mathbb{R}^m} \left\{ h^T\nabla^2_{yy} L_{\min}(\bar x,\bar y,\alpha,\beta)h +2 u^T\nabla^2_{xy} L_{\min}(\bar x,\bar y,\alpha,\beta)h  + u^T\nabla^2_{xx} L_{\min}(\bar x,\bar y,\alpha,\beta)u \right\}\\
& = \nabla^2_{(x,y)} L_{\min}(\bar x,\bar y,\alpha,\beta)((u,h^*),(u,h^*)) \\
& = u^T [\nabla_{xx}^2 L_{\min}-\nabla_{xy}^2 L_{\min}(\nabla_{yy}^2 L_{\min})^{-1}\nabla_{yx}^2 L_{\min}](\bar x,\bar y,\alpha,\beta) u \\
& =u^T \left(\sum_{i=1}^{p} \alpha_i \nabla^2 \phi_i(\bar x) +[\nabla_{xx}^2 L_{\max}-\nabla_{xy}^2 L_{\max}(\nabla_{yy}^2 L_{\max})^{-1}\nabla_{yx}^2 L_{\max}](\bar y,\beta;\bar x) \right)u.
\end{aligned}
\end{eqnarray*}
Hence (\ref{eqn5.6}) follows from Theorem \ref{cor5.1} (b)(ii).
\end{proof}

\subsection{Unconstrained case}
In this section we consider the unconstrained minimax problem
\begin{equation}\label{unconstrained}
  \min_{x\in \mathbb{R}^n}\max_{y\in \mathbb{R}^m} f(x,y).
  \end{equation}

In \cite[Propositions 19 and 20]{JNJ20}, Jin et al. gave some second-order optimality conditions for local minimaxity for the unconstrained minimax problem when $f$ is twice differentiable. For the case where $f$ is nonsmooth,  Theorem \ref{main}  gives necessary and sufficient optimality for the calm local minimaxity. An illustrative example is given as follows.

\begin{example}\cite[Example A.1]{JC22}\label{example-non}
Consider
$$
\min _{x \in \mathbb{R}} \max _{y \in \mathbb{R}} f(x, y):=-|x|^9+\frac{3}{5}|x|^3|y|^3-|y|^5 .
$$
We will show that $(0,0)$ is a calm local minimax point and  the necessary optimality condition in Theorem \ref{main} holds.

 Take $\tau(\delta)=\frac{3}{5}(\sqrt{\delta})^3$. Then for any $|x| \leq \delta$ and $|y| \leq \delta$ with sufficiently small $\delta \in(0,1)$ we have
$$
-|y|^5=f(0, y) \leq f(0,0) \leq \max _{y \in[-\tau(\delta), \tau(\delta)]}-|x|^9+\frac{3}{5}|x|^3|y|^3-|y|^5=-|x|^9+\frac{2}{5}\left(\frac{3}{5}\right)^4(\sqrt{|x|})^{15}
$$
where $\pm \frac{3}{5}(\sqrt{|x|})^3$ is the maximizer of the above maximization problem. Since $\tau$ is calm at $0$, $(0,0)$ is a calm local minimax point.

Obviously, $f$ is not differentiable at $(0,0)$.
{We first verify that the assumptions in Theorem \ref{main} hold at $(0,0)$. Then, we show that the necessary optimality conditions in Theorem \ref{main}(b) hold at $(0,0)$.

{(i) Denote $\varphi(\alpha,\beta):=-\alpha^9+\frac{3}{5}\alpha^3\beta^3-\beta^5, g(x):=|x|, g(y):=|y|$, and $(\bar x,\bar y)=(0,0)$. Applying Proposition \ref{chainrule}, we know that $f$ is twice semidifferentiable at $(0,0)$ and the separation property holds for the subderivative holds at $(0,0)$. Moreover,}
for any $u \in \mathbb{R}$ and $h \in \mathbb{R}$,
\begin{equation*}
\mathrm{d} f(0,0)(u,h) =\nabla \varphi(g(\bar x),g(\bar y))^T (g'(\bar x;u),g'(\bar y;h)) = 0,
\end{equation*}
\begin{equation}\label{dx-example}
\mathrm{d}_x f(0,0)(u)=\nabla_\alpha \varphi (g(\bar x),g( \bar y))g'(\bar x;u)=0,
\end{equation}
\begin{equation}\label{dy-example}
\mathrm{d}_y f(0,0)(h)=\nabla_\beta \varphi (g(\bar x),g( \bar y))g'(\bar y;h)=0,
\end{equation}
\begin{equation*}
\begin{aligned}
\mathrm{d}^2 f(0,0)(u,h)& = (g'(\bar x;u),g'(\bar y;h))^T \nabla^2 \varphi(g(\bar x),g(\bar y))(g'(\bar x;u),g'(\bar y;h))=0,
\end{aligned}
\end{equation*}
$$
\mathrm{d}_{yy}^2 f(0,0)(h)=g'(\bar y;h)^T \nabla_{yy}^2 \varphi(g(\bar x),g(\bar y))g'(\bar y;h)=0.
$$

(ii) We show that for any $u\in T_{X}(0) \cap \{u' | {\rm d}_x f(0,0)(u') = 0\},$ the value ${\rm d}^2 \delta_X(0;-{\rm d}_xf(0,0))(u)$ is finite. This can be seen from the fact that for any $u \in \mathbb{R}$,
$$
{\rm d}^2 \delta_X(0;-{\rm d}_xf(0,0))(u)= \liminf\limits_{t \downarrow 0,u' \to u} \frac{t{\rm d}_xf(0, 0)(u')}{\frac{1}{2}t^2}=0.
$$

(iii) We show that the necessary optimality conditions in Theorem \ref{main} (b) hold at $(0,0)$. By \eqref{dx-example} and \eqref{dy-example}, we know that $T_X(\bar x) \cap \{u' | {\rm d}_x f(\bar x,\bar y)(u') = 0\} =T_ Y({\bar y})\cap \{h' | {\rm d}_y f\left(\bar x, \bar y\right) h' = 0\}=\mathbb{R}$. Besides, for any $h \in \mathbb{R}$,

$$
{\rm d}^2 \delta_Y(0;{\rm d}_yf(0,0))(h)= \liminf\limits_{t \downarrow 0,h' \to h} \frac{-t{\rm d}_yf(0, 0)(h')}{\frac{1}{2}t^2}=0.
$$

Thus, for any $u, h \in \mathbb{R}$,
 \begin{equation*}\label{}
{\rm d}^2 f(0,0)(u,h) +{\rm d}^2 \delta_X(0;-{\rm d}_xf(0,0))(u)-{\rm d}^2 \delta_Y(0;{\rm d}_yf(0,0))(h) = 0.
\end{equation*}
Besides, for any $h\in \mathbb{R}$, we have
$$
{\rm d}^2_{yy} f(0,0)(h) -{\rm d}^2 \delta_Y(0;{\rm d}_yf(0,0))(h) \leq 0.
$$}
\end{example}

{Theorem \ref{second-dual}} has the following immediate corollary. We state it here for the convenience of comparison with the existing results.
\begin{theorem}[second-order  conditions for the unconstrained smooth case]\label{uncons}
 Suppose that $f$ is twice differentiable at $(\bar x, \bar y)$.
\begin{itemize}
\item[(a)]
Suppose that {$\nabla f(\bar x,\bar y)=0$,} $\nabla_{yy}^2 f(\bar x,\bar y) \prec 0$, and for any $u\in \mathbb{R}^n \setminus \{0\}$, there exists $h\in \mathbb{R}^m$ such that
\begin{equation*}\label{}
{\nabla^2} f(\bar x,\bar y)((u,h),(u,h)) >0.
\end{equation*}
Then $(\bar x,\bar y)$ is a calm local minimax point to problem \eqref{unconstrained}.
\item[(b)] Let $(\bar x,\bar y)$ be a calm local minimax point to problem \eqref{unconstrained}. Then {$\nabla f(\bar x,\bar y)=0$,} $\nabla_{yy}^2 f(\bar x,\bar y) \preceq 0$ and for any $u\in \mathbb{R}^n$, there exists $h \in \mathbb{R}^m$  such that
 \begin{equation}\label{uncons-nec}
{\nabla^2} f(\bar x,\bar y)((u,h),(u,h))  \geq 0.
\end{equation}
\end{itemize}
\end{theorem}
We can use the above necessary condition to verify that $(0,0)$ is not a calm local minimax point for the problem in Example \ref{ex3.1}, since
$\nabla f(0,0)=(0,0)$, $\nabla^2_{yy} f(0,0)= 0$ and
for any $u\not =0$ and any $h \in \mathbb{R}$,
$
\nabla^2 f(0,0)((u,h),(u,h)) =-2u^2<0.
$

By Corollary \ref{local-to-lip2}, when $X=\mathbb{R}^n, Y=\mathbb{R}^m$ and $(\bar x,\bar y)$ is a local minimax point with $\nabla^2_{yy} f(\bar x,\bar y)\prec 0$, we have that  $(\bar x,\bar y)$ is a calm local minimax point. Letting $\nabla_{yy}^2 f(\bar x,\bar y) \prec 0$ and $h^*:=-\nabla_{yy}^2 f(\bar x,\bar y)^{-1}\nabla_{xy}^2 f(\bar x,\bar y)^T u$, the result in \cite[Propositions 19 and 20]{JNJ20} can be recovered from Theorem \ref{uncons}.

\begin{corollary}\label{unconstrained-nec}
Let $(\bar x,\bar y)\in \mathbb{R}^n\times \mathbb{R}^m$.
\begin{itemize}
\item[(a)] Suppose that the following first-order and second-order sufficient conditions hold.
\begin{equation*}\label{}
{\nabla_x f(\bar x, \bar y)=\nabla_y f(\bar x, \bar y)=0}, \quad \nabla_{yy}^2 f(\bar x,\bar y) \prec 0,
   \end{equation*}
\begin{equation*}\label{}
[\nabla_{xx}^2 f-\nabla_{xy}^2 f(\nabla_{yy}^2 f)^{-1}\nabla_{yx}^2 f](\bar x,\bar y)  \succ 0.
\end{equation*}
Then $(\bar x,\bar y)$ is a (calm) local minimax point to the problem \eqref{unconstrained}.
\item[(b)] Let $(\bar x,\bar y)$ be a local minimax point for problem \eqref{unconstrained}. Suppose that $\nabla_{yy}^2 f(\bar x,\bar y) \prec 0$. Then
\begin{equation*}\label{}
[\nabla_{xx}^2 f-\nabla_{xy}^2 f(\nabla_{yy}^2 f)^{-1}\nabla_{yx}^2 f](\bar x,\bar y) \succeq 0.
\end{equation*}
\end{itemize}
\end{corollary}
\begin{proof}
For each $u \in  \mathbb{R}^n \setminus \{0\}$, let $h^*:=-\nabla_{yy}^2 f(\bar x,\bar y)^{-1}\nabla_{xy}^2 f(\bar x,\bar y)^T u$. {With Theorem \ref{uncons}, the rest of the proof can be given similarly to the proof for Corollary \ref{Thm3.1D} (ii).}
\end{proof}

Note that  Corollary \ref{unconstrained-nec} (b)  is only applicable when $\nabla_{yy}^2 f(\bar x,\bar y) \prec 0$. This is restrictive. In contrast, Theorem \ref{uncons} (b) is applicable even when $\nabla_{yy}^2 f(\bar x,\bar y)$ is only positive semidefinite. The following example  from \cite[Example 3.21]{JC22}, shows that there exist minimax problems for which the optimality conditions in Theorem \ref{uncons}(b) are  applicable while  the one in Corollary \ref{unconstrained-nec} (b) is not.

\begin{example}\label{example=0}
Consider
$$
\min _{x \in \mathbb{R}} \max _{y \in \mathbb{R}} f(x, y):=-x^4+4 x^2 y^2-y^4 .
$$
We first exam that $(\bar x,\bar y) = (0,0)$ is a calm local minimax point, and thus is also a local minimax point.
Let $\tau(\delta)=\sqrt{2} \delta$ and $\delta_0=\frac{\sqrt{2}}{2}$. Then, for any $\delta \in\left(0, \delta_0\right]$ and any $(x, y) \in \mathbb{R}^2$ satisfying $|x| \leq \delta$ and $|y| \leq \delta$, we have
$$
-y^4=f(0, y) \leq f(0,0) \leq \max _{|y'| \leq \tau(\delta)} f\left(x, y^{\prime}\right)=3 x^4.
$$
Since $\nabla^2_{yy} f(\bar x, \bar y) = 0$, the optimality conditions in Corollary \ref{unconstrained-nec}(b) cannot be applied. However, we can easily verify that necessary optimality conditions in Theorem \ref{uncons}(b) holds at $(\bar x, \bar y)$ since
\begin{equation*}
\nabla^2 f(\bar x, \bar y) =
\begin{pmatrix}
0 & 0  \\
0 & 0
\end{pmatrix}.
\end{equation*}
Besides, it is easy to verify that $(0,0)$ is the unique global minimax point (by the definition) and the unique local minimax point (by using the first-order condition $\nabla f(x,y)=0$). However, $(0,0)$ is not a local Nash equilibrium since $f(0,0)=0 \geq f(x,0) = -x^4$.
\end{example}
When reducing to the unconstrained case, the second-order necessary and sufficient optimality conditions for local minimax points in \cite[Theorems 3.1 and 3.2]{DaiZh-2020} are the same as the result in Corollary \ref{unconstrained-nec}. We have shown in Example \ref{example=0} that our necessary optimality condition extends the result in Corollary \ref{unconstrained-nec} (a), and thus extends the result in \cite[Theorem 3.1]{DaiZh-2020}.

In \cite{JC22}, Jiang and Chen studied local minimax points and derived some necessary optimality conditions. We state their result \cite[Theorem 3.17]{JC22} for the unconstrained case so that we can compare it with Theorem \ref{uncons} (b).

\begin{proposition}[\text{\cite[Theorem 3.17]{JC22}}]\label{Theorem 3.17JC22}
 Let $(\bar x,\bar y)$ be a local minimax point for the unconstrained minimax problem \eqref{unconstrained} where $f$ be twice continuously differentiable. Then it holds that $\nabla f(\bar x,\bar y)=0, \nabla_{yy}^2 f(\bar{x}, \bar{y}) \preceq 0$, and
\begin{equation} \label{eqn5.9}u^T\nabla_{x x}^2 f(\bar{x}, \bar{y}) u\ \geq 0 \text { for all } u \in \mathrm{cl}\left\{\bar{u}\,|\, \exists \delta>0, \bar{u} \in C_{\min}(\bar x,y')\ \forall y^{\prime} \in \mathbb{B}_\delta(\bar{y})\right\},
\end{equation}
where
$$C_{\min}(\bar x,y'):= \{u\in \mathbb{R}^n \,|\,  \nabla_x f\left(\bar x, y'\right)u = 0\}.$$
\end{proposition}

In general we can not directly show that our condition (\ref{uncons-nec}) is sharper than (\ref{eqn5.9}). However when the set
$\mathrm{cl}\left\{\bar{u}\,|\, \exists \delta>0, \bar{u} \in C_{\min}(\bar x,y')\ \forall y^{\prime} \in \mathbb{B}_\delta(\bar{y})\right\}$ in  (\ref{eqn5.9}) is equal to $\{0\}$, our condition  (\ref{uncons-nec}) is shaper.
 To be more specific,
consider the problem
$$\min_{x\in \mathbb{R}} \max_{y\in \mathbb{R}} -y^2+xy+x^3+x^4.$$
Using the first-order condition $\nabla f(\bar x, \bar y)=0$, there are three stationary points $(0,0)$, $(-\frac{1}{2},-\frac{1}{4})$, and $(-\frac{1}{4},-\frac{1}{8})$.
Using the second-order sufficient condition in Corollary \ref{unconstrained-nec} (a), we can identify that $(0,0)$ and $(-\frac{1}{2},-\frac{1}{4})$ are local minimax points.

Using the second-order necessary condition in Corollary \ref{unconstrained-nec} (b), we can identify that $(-\frac{1}{4},-\frac{1}{8})$ is not a calm local minimax point. The problem is a smooth nonconvex-strongly-concave minimax problem satisfying the conditions required in Corollary \ref{local-to-lip2}, and thus local minimax points are identical to calm local minimax points.

However, for $(\bar x,\bar y) = (-\frac{1}{4},-\frac{1}{8})$ and any $\delta >0, y' \in \mathbb{B}_{\delta}(\bar y)$ with $y' \not= \bar y$, we have $y' +  3 \bar x^2 + 4 \bar x^3 =y' + \frac{1}{8}\not=0$, and thus
$$C_{min}(\bar x,y')=\{u\in \mathbb{R} \, | \,(y' +  3 \bar x^2 + 4 \bar x^3)u=0\}=\{0\}.$$
Therefore,
$$\operatorname{cl}\{\bar u \,|\, \exists \delta >0, \bar u \in C_{min}(\bar x,y')\ \forall y' \in \mathbb{B}_{\delta}(\bar y)\} = \operatorname{cl}\{0\} = \{0\}, $$
which means that the second-order condition for a local minimax \eqref{eqn5.9} holds.
Therefore the second-order condition for a local minimax \eqref{eqn5.9} cannot rule out the possibility that  $(-\frac{1}{4},-\frac{1}{8})$ is not a local minimax point.

\end{document}